\documentclass[12pt]{amsart}
\usepackage[utf8]{inputenc}
\usepackage[T1]{fontenc}
\usepackage{lmodern}
\usepackage{color,graphicx,array, amssymb, amscd,slashed,bm}
\usepackage{comment}
\usepackage{hyperref}
\newtheorem{Theorem}{Theorem}[section]
\newtheorem{Lemma}[Theorem]{Lemma}
\newtheorem{Proposition}[Theorem]{Proposition}
\newtheorem{Corollary}[Theorem]{Corollary}
\theoremstyle{definition}
\newtheorem{Definition}[Theorem]{Definition}

\newtheorem{Remark}[Theorem]{Remark} 

\numberwithin{equation}{section}
\newcommand{\R}{\mathbb R}

\newcommand{\C}{\mathbb C}

\newcommand{\Uone}{{\rm U}_1}
\newcommand{\id}{\operatorname{id}}

\newcommand{\la}{\langle}
\newcommand{\ra}{\rangle}

\newcommand{\SL}{{\rm SL}(2, \mathbb C)}

\newcommand{\SU}{{ \rm SU(2)}}
\newcommand{\SO}{{ \rm SO(4)}}
\newcommand{\U}{{ \rm U(1)}}
\newcommand{\Sp}{{ \rm Sp(1)}}
\newcommand{\bsi}{{  \boldsymbol{i}}}
\newcommand{\bsj}{{  \boldsymbol{j}}}
\newcommand{\bsk}{{  \boldsymbol{k}}}
\newcommand{\bsH}{{  \boldsymbol{H}}}
\newcommand{\bsR}{{  \boldsymbol{R}}}
\newcommand{\bsC}{{  \boldsymbol{C}}}

\renewcommand{\Re}{\operatorname {Re}}
\renewcommand{\Im}{\operatorname {Im}}

\newcommand{\red}[1]{{\leavevmode\color{red}{#1}}} 

\usepackage{amsmath}	
\usepackage{cases}
\begin{document}

\title[]{Minimal Lagrangian surfaces in 
the two dimensional complex quadric via the 
loop group method}
 \author[S.-P.~Kobayashi]{Shimpei Kobayashi}
 \address{Department of Mathematics, Hokkaido University, 
 Sapporo, 060-0810, Japan}
 \email{shimpei@math.sci.hokudai.ac.jp}
 \author[S.~Zeng]{Sihao Zeng}
 \address{Department of Mathematics, Hokkaido University, 
 Sapporo, 060-0810, Japan}
 \email{sihao.zeng.f1@elms.hokudai.ac.jp}
 \thanks{The first named author is partially supported by JSPS 
 KAKENHI Grant Number JP22K03304.
 The second named author is  supported 
 by Special Program: International Graduate Course for Data-driven and Hypothesis-driven Science (IGC-DHS)}
 \date{\today}
\pagestyle{plain}
\begin{abstract}We develop a loop group (DPW-type) representation for minimal Lagrangian surfaces in the complex quadric
$Q_{2} \cong \mathbb S^{2}\times \mathbb S^{2}$, formulated via a flat family of connections
$\{\nabla^\lambda\}_{\lambda\in \mathbb S^{1}}$ on a trivial bundle.
We prove that minimality is equivalent to the flatness of $\nabla^\lambda$ for all $\lambda$,
describe the associated isometric $\mathbb S^{1}$-family, and establish a precise correspondence with
minimal surfaces in $\mathbb S^{3}$ through their Gauss maps.
Our framework unifies and streamlines earlier constructions (e.g., Castro--Urbano) and yields
explicit families including $\mathbb R$-equivariant, radially symmetric, and trinoid-type examples.
\end{abstract}
\maketitle
\section*{Introduction}
Minimal surfaces in K\"ahler manifolds \cite{CW1983} play a fundamental role at the interface of complex and Riemannian geometry, and in the Lagrangian setting, they further connect to symplectic geometry \cite{O1994}. 
 Thus minimal Lagrangian surfaces in two-dimensional K\"ahler manifolds are of particular importance.
 Furthermore, the complex projective plane $\mathbb{C}P^2$ and 
 the product $\mathbb{S}^{2} \times \mathbb{S}^{2}$
 are the only two-dimensional Hermitian symmetric spaces of compact type, and both are 
 K\"ahler--Einstein surfaces.
 Indeed,  I. Castro and F. Urbano  studied  minimal Lagrangian surfaces in 
 $\mathbb{C} P^2$ \cite{CU1994}
 and in $\mathbb{S}^{2} \times \mathbb{S}^{2}$ \cite{CU}, revealing a variety of fundamental properties that highlight their geometric significance. Since $\mathbb{S}^{2} \times \mathbb{S}^{2}$
 is isometric to  the complex quadric $Q_2$  in the complex projective space 
 $\mathbb{C}P^{3}$, minimal (Lagrangian) surfaces in $Q_2$ are 
 also of interest in algebraic geometry; see \cite{JW2013, WX}. 
 
 Meanwhile, harmonic maps from a surface into symmetric spaces can be formulated within the framework of integrable systems, via  the construction of a family of flat connections on a trivial 
 principal $G$-bundle;  see the loop group method of J.~Dorfmeister, F.~Pedit, and H.~Wu \cite{DPW} by exploiting loop group decompositions of infinite-dimensional Lie groups (the DPW method).

This method has been successfully applied to Hamiltonian stationary Lagrangian surfaces in $\mathbb{C}^{2}$ and $\mathbb{C}P^{2}$ \cite{HR2002, HR2005}, 
as well as to minimal Lagrangian surfaces in $\mathbb{C}P^{2}$ \cite{ DKM2020, DM2021, DM2025}. It is therefore natural to extend this approach to minimal Lagrangian surfaces in the complex quadric
$Q_2$, thereby unveiling their structure via integrable-systems techniques.

In this paper, we begin with a Lagrangian conformal immersion 
\(f : M  \to Q_{2} \subset \mathbb{C}P^3\) from a Riemann surface $M$, 
together with a horizontal lift 
\(\mathfrak{f} : \mathbb{D} \subset M \to \mathbb{S}^{7} \subset \mathbb{C}^{4}\). 
We then consider a natural \(\mathrm{SO}(4)\)-moving frame and its Maurer--Cartan form 
\(\omega\), which takes values in 
\(\mathfrak{so}(4) = \mathrm{Lie}(\mathrm{SO}(4))\). Although the resulting structure equations are involved, we express them in terms of the invariants of $\mathfrak f$:
\[
e^{u}:= \langle \mathfrak f_{z},\overline{\mathfrak f_{z}}\rangle,\quad
\alpha:=\langle \mathfrak f_{z},\mathfrak f_{z}\rangle,\quad
\beta:=\langle \mathfrak f_{z},\mathfrak f_{\bar z}\rangle\quad\mbox{and}\quad
\phi:=e^{-u}\langle \mathfrak f_{z\bar z},\overline{\mathfrak f_{\bar z}}\rangle,
\]
where $\langle \,, \,\rangle$ denotes the scalar product in $\mathbb C^4$ and the 
subscripts $z$ and $\bar z$ are the derivatives with respect to the conformal coordinate 
$z \in \mathbb D$ and its conjugate.
 Note that $ds^{2}=2e^{u}\,\mathrm{d}z\mathrm{d}\bar z$ is the induced metric and 
 $\alpha \, \mathrm{d} z^2$  is the quadratic differential
 and the vanishing of \(\phi \, \mathrm{d}z\) is equivalent to the minimality of 
\(f\) \cite[Theorem 3.2]{WX}.

Since \(Q_2\) is a Hermitian symmetric space 
$\SO/ \rm{SO}(2) \times \rm{SO}(2)$, we can introduce a special 
family of connection $1$-forms \(\mathrm{d} + \omega^{\lambda}\), 
parametrized by \(\lambda \in \mathbb{S}^{1}\), such that 
\(\omega^{\lambda} |_{\lambda=1}\) becomes the Maurer--Cartan form of 
\(\mathfrak{f}\). 
In Theorem~\ref{Thm1}, we show that the flatness of 
\(\mathrm{d} + \omega^{\lambda}\) for all \(\lambda \in \mathbb{S}^{1}\) is 
equivalent to minimality of a Lagrangian surface. 
Moreover, minimality is further equivalent to the conditions that the quadratic differential 
 \(\alpha \, \mathrm{d} z^2\) is holomorphic and that the argument of $\beta$ is constant.
 Then we can take a new local horizontal lift 
 so that the complex 
function \(\beta\) becomes a non-negative real function, according to 
Theorem~\ref{Thm1}. 

Next, by choosing a suitable gauge transformation, we show the existence of a family 
of isometric minimal Lagrangian immersions 
\(\{ f^{\lambda} \}_{\lambda \in \mathbb{S}^{1}}\) 
(Theorem~\ref{thm2}). 
Moreover, it turns out that the induced metric can be expressed in 
terms of a function \(\hat u\) and the holomorphic function \(\alpha\) as
\[
 2 e^{u}\, \mathrm{d} z\mathrm{d}\bar{z}
   = \bigl(e^{\hat u} + |\alpha|^{2} e^{-\hat u}\bigr)\, \mathrm{d} z\mathrm{d}\bar{z}.
\]
 where $\hat u$ satisfies the sinh-Gordon equation:
\[\hat u_{z \bar z} + e^{\hat u} - |\alpha|^{2} e^{-\hat u} = 0,\]
It turns out that the metric of $f$ is 
induced by half the Sasakian metric of the unit tangent bundle $\mathrm{U}\mathbb S^3$
of a minimal surface in the unit three-sphere $\mathbb S^3$.
Thus we establish a  correspondence between the minimal Lagrangian 
surfaces in \(Q_2\) and the Gauss maps of minimal surfaces in \(\mathbb{S}^{3}\) 
(Theorem~\ref{Thm corresponding}). 
This correspondence was previously shown under the 
additional condition \(\alpha \neq 0\) in \cite{CU}.
More generally, such correspondences of hypersurfaces in unit spheres and 
Lagrangian immersions in Grassmannians have been studied in \cite{P1997, VW2020}.
Since the structure equation, that is the sinh-Gordon equation, of a minimal Lagrangian surface in \(Q_{2}\) 
coincides with those of a harmonic map from a Riemann surface into 
the unit two-sphere \(\mathbb{S}^{2}\), this suggests that the loop group method can be applied 
to minimal Lagrangian surfaces in \(Q_{2}\) through harmonic maps into 
\(\mathbb{S}^{2}\). 
Indeed, by the well-known Lie group isomorphism
\[
 \mathrm{SO}(4) \cong \bigl(\mathrm{SU}(2) \times \mathrm{SU}(2)\bigr) / \boldsymbol{Z}_{2},
\]
this can be achieved (see Section~\ref{sc:DPW}).

In Section~\ref{sc:Ex}, we construct several examples of minimal Lagrangian 
surfaces via the loop group method: a family of totally geodesic spheres and totally 
geodesic tori. 
As new examples, we also construct families of $\mathbb R$-equivariant surfaces, radially 
symmetric surfaces, and trinoids which are minimal Lagrangian immersions from 
a thrice-punctured sphere.

\subsection*{Main results} 
We demonstrate for the first time that the loop group method 
can be applied to minimal Lagrangian geometry in the complex quadric 
\(Q_{2} \subset \mathbb{C}P^3 \), yielding new explicit families such as $\mathbb R$-equivariant and radially 
symmetric surfaces and trinoids.  Finally, we summarize the main results.
\begin{description}
  \item[(R1)] Characterization (Theorem \ref{Thm1}).
  For \(f:\mathbb{D}\to Q_{2}\),
  \(f\) minimal \(\iff\) \(\mathrm{d}+\omega^{\lambda}\) are flat for all \(\lambda\in\mathbb{S}^{1}\)
  \(\iff\) \(\alpha\) holomorphic and \(\arg\beta\) constant.
  After a normalized horizontal lift, we can assume \(\beta\ge0\); then \(\hat u\) solves
  the sinh-Gordon equation. 

  \item[(R2)] Associated family   (Theorem \ref{thm2}).
  An isometric \( \mathbb{S}^{1}\)-family \(\{f^{\lambda}\}\) exists with the holomorphic quadratic 
  differential  \(\alpha^{\lambda}=\lambda^{-2} \alpha\).

  \item[(R3)]  Correspondence with \(\mathbb S^{3}\)   (Theorem \ref{Thm corresponding}).
  Minimal Lagrangian surfaces in \(Q_{2}\) correspond to the Gauss maps of minimal surfaces in \( \mathbb{S}^{3}\),
  without assuming \(\alpha\neq0\).

  \item[(R4)] The DPW method and new examples   (Sections \ref{sc:DPW} and \ref{sc:Ex}).
  The method has been established in Section \ref{sc:DPW}. In Section \ref{sc:Ex},
  totally geodesic spheres/tori, and new families of examples: $\mathbb R$-equivariant surfaces, radially symmetric surfaces, and trinoids, obtained by appropriate choices of holomorphic potentials.
  \end{description}

{\bf Acknowledgments:}
 We would like to express our sincere gratitude to Prof. Jun-ichi Inoguchi for inspiring our interest in this problem. We also would like to thank Prof. David Brander and Prof. Peng Wang for helpful discussions.
\section{Minimal Lagrangian surfaces in $Q_2$, the family of flat connections and minimal surfaces in $\mathbb S^3$}
In this section, we consider a Lagrangian immersion $f : M \to Q_{2}$ from a  Riemann surface $M$ into the complex quadric $Q_2$ in $\mathbb{C}P^{3}$.
 We explicitly compute the $\mathrm{SO}(4)$-moving frame and its Maurer-Cartan form and equation
 for $f$ and then we study the special family of connections 1-forms $\text{d} + \omega^{\lambda}$, parametrized by $\lambda \in \mathbb{S}^{1}$ such that $\omega^{\lambda}|_{\lambda=1}$ is the Maurer-Cartan form of 
$f$.  We show that the flatness of the whole family $\text{d} + \omega^{\lambda}$ is equivalent to the minimality of the Lagrangian surfaces $f$, Theorem \ref{Thm1}, and thus there exists a
 one-parameter family of minimal Lagrangian surfaces, Theorem \ref{thm2}. Furthermore, we show a correspondence between minimal Lagrangian surfaces and minimal surfaces in $\mathbb S^3$, Theorem \ref{Thm corresponding}.

\subsection{Lagrangian surfaces in the two dimensional quadric $Q_2$}
 Let $( N , \omega )$ be a K\"ahler manifold of $\text{dim}_{\mathbb{C}}N = n$ with the 
 K\"ahler form $\omega$. An immersion $f : M \to N$ from an $m$-dimensional manifold $M$ into $N$ is said to be $\textit{totally~real}$ if $f^{\ast}\omega = 0$. 
 In particular, a totally real immersion $f$ is said to be $\textit{Lagrangian}$ if $m = n$.
  
Let $\mathbb{C}^{n}$ be the $n$-dimensional complex Euclidean space with the complex bilinear form
$\la \,,\, \ra$ defined by
 \begin{equation}\label{eq:sclar}
     \left \la \boldsymbol{z} , \boldsymbol{w} \right \ra = \sum^{n}_{k = 1}z_{k}w_{k}, \quad \text{for}~\boldsymbol{z} = \left( z_{1}, \ldots, z_{n} \right), \, \boldsymbol{w} = \left( w_{1}, \ldots, w_{n}\right) \in \mathbb{C}^{n}.
 \end{equation}
 The standard Hermitian inner product $(~,~)$ on $\mathbb{C}^{n}$ is given by $\left( \boldsymbol{z}, \boldsymbol{w} \right) = \left \la \boldsymbol{z} , \bar{\boldsymbol{w}} \right \ra$, where $\bar{\boldsymbol{w}}$ is the conjugate of $\boldsymbol{w}$. In this paper, three-dimensional complex projective space 
 is denoted by $\mathbb{C}P^{3} : = \mathbb C^4\setminus \{0\}/\sim$, and 
  the complex quadric $Q_2$ in $\mathbb{C}P^{3}$ can be realized by
\begin{equation}\label{eq:Q2}
Q_{2} = \left \{ \left [ \boldsymbol{z} \right ] \in \mathbb{C}P^{3} \mid  \boldsymbol{z} \in 
\mathbb C^4,  \la 
 \boldsymbol{z}, \boldsymbol{z}\ra = 0 \right \}.
\end{equation}
 We introduce the Fubini-Study metric of constant holomorphic sectional curvature $4$
 on $\mathbb{C}P^{3}$, which can be 
 explicitly written by
 \[ ds^2 = (dZ - \langle dZ,  \bar Z \rangle Z ) \otimes (d\bar Z-\langle d \bar Z , Z\rangle \bar Z), \]
 where $Z$ in $\mathbb S^7 \subset \mathbb C^4$ is a local holomorphic section of the tautological bundle, 
 and thus $\mathbb{C}P^{3}$ becomes a K\"ahler manifold. 
 The metric on $Q_2$ is naturally induced by the Fubini-Study metric, and thus $Q_2$ becomes a K\"ahler surface. 
 It is well known that $Q_{2}$ is isometric to $\mathbb S^2 \times  \mathbb S^2$, 
 where the curvatures of the two $2$-sphere $\mathbb S^2$ are normalized to $4$, \cite{CU, WV2021}.

 Let $f : M \to Q_{2}$ be a Lagrangian conformal immersion from a Riemann surface $M$ into $Q_{2}$. Moreover, let $\mathbb{D} \subset M$ be a simply connected domain with conformal coordinate $z = x + iy$. 
 Then the induced metric on $\mathbb{D}$ can be computed as
 \[
 ds^{2}_{M} = 2e^{u}\text{d}z\text{d}\bar{z}.
 \]
 Let $\mathfrak{f}: \mathbb{D} \to \mathbb{S}^{7} \subset \mathbb C^4$ be a local lift of $f$, i.e. $f = \pi \circ \mathfrak{f}$, where $\pi : \mathbb{S}^{7} \to \mathbb{C}P^{3}$ is the Hopf fibration. In fact, the projection $f$ can be realized as $[\mathfrak{f}]$.
 Since $f$ is conformal and Lagrangian, we obtain
 \begin{equation*}
 \langle  \mathfrak{f}_{z}, \overline{\mathfrak{f}_{\bar{z}}} \rangle  = 0, \quad \langle \mathfrak{f}_{z}, \overline{\mathfrak{f}_{z}} \rangle = \langle \mathfrak{f}_{\bar{z}}, \overline{\mathfrak{f}_{\bar{z}}} \rangle = e^{u}.
 \end{equation*}
 Here $\partial_z = \tfrac12 (\partial_x - i \partial_y)$ and 
 $\partial_{\bar z} = \tfrac12 (\partial_x + i \partial_y)$ are the complex differentiations,
 and the imaginary unit  $\sqrt{-1}$ is denoted by $i$.
 Since $f(M) \subset Q_{2}$, we have
 \begin{equation*}
     \langle  \mathfrak{f}, \mathfrak{f} \rangle  = 0, \quad \langle \mathfrak{f}_{z}, \mathfrak{f}  \rangle = \left\la \mathfrak{f}_{\bar{z}}, \mathfrak{f} \right\ra = 0.
 \end{equation*}
 
 \begin{Definition}
 If a local lift $\mathfrak{f}$ defined above satisfies 
 \begin{equation*}
     \langle \mathfrak{f}_{z} , \bar{\mathfrak{f}} \rangle = \langle \mathfrak{f}_{\bar{z}} , \bar{\mathfrak{f}} \rangle = 0, 
 \end{equation*}
  then we call $\mathfrak{f}$ a $\textit{horizontal~lift}$.
 \end{Definition}
 Since horizontal lifts of $f$ are not unique, we fix one horizontal lift $\mathfrak f$ for time being.
Since the special orthogonal Lie group $\SO = \left\{  A \in M_{4 \times 4} (\R) \mid  A^T = A^{-1},
 \det A =1\right\}$ acts transitively  on $Q_2$ as the orientation-preserving isometry, $Q_{2}$ is isomorphic to the symmetric space: 
 \begin{equation}\label{eq:symQ2}
 Q_2 = \SO/ \rm{SO}(2) \times \rm{SO}(2),
 \end{equation}
  see \cite{Y1989}. Indeed, by choosing the 
  involution $\sigma = \operatorname{Ad} \operatorname{diag}(1, 1, -1, -1)$ on $\SO$,
  the fixed point set of $\sigma$ is exactly $ \rm{SO}(2) \times \rm{SO}(2)$.
    Let $f : M \to Q_2=\SO/ \rm{SO}(2) \times \rm{SO}(2)$ and let $\mathcal{F}$ be a local lift of $f$ as
 \begin{equation*}
     \mathcal{F}:= \left( \frac{1}{\sqrt{2}}\left( \mathfrak{f} + \bar{\mathfrak{f}} \right), -\frac{i}{\sqrt{2}}\left( \mathfrak{f} - \bar{\mathfrak{f}} \right), e_{3}, e_{4} \right): \mathbb{D} \subset M \to \rm{SO}(4),
 \end{equation*}
 where $\mathfrak{f}$ is a horizontal lift defined above, and $e_3$ and $e_4$ are defined by 
     \begin{equation*}
     e_{3} = \frac{\mathfrak{f}_{z} + \overline{\mathfrak{f}_{z}}}{\sqrt{2e^{u} + \alpha + \bar{\alpha}}}, \quad e_{4} = -\frac{i\left \{\mathfrak{f}_{z}\left( e^{u} + \bar{\alpha} \right) - \overline{\mathfrak{f}_{z}}\left( e^{u} + \alpha \right) \right \} }{\sqrt{\left( 2e^{u} + \alpha + \bar{\alpha} \right)\left( e^{2u} - \alpha\bar{\alpha} \right)}},
 \end{equation*}
 with 
 \begin{equation}\label{eq:alpha}
 \alpha := \la \mathfrak{f}_{z}, \mathfrak{f}_{z} \ra.
 \end{equation}
 By direct computation $e^{2u} - \alpha\bar{\alpha} \geq  0$ and 
 we assume that $e^{2u} - \alpha\bar{\alpha} > 0$ for the time being so that 
 the frame $\mathcal F$ is well-defined.
 Note that this assumption is not necessary for minimal Lagrangian surfaces, as we will see later,
 see Remark \ref{Rm:assumption}.
 Its Maurer-Cartan form can be computed as follows:
 \begin{equation}
     \omega = \mathcal{F}^{-1}\text{d}\mathcal{F} =  \mathcal{F}^{-1}\mathcal{F}_{z}\text{d}z + \mathcal{F}^{-1}\mathcal{F}_{\bar{z}}\text{d}\bar{z} = \mathcal{U}\text{d}z + \mathcal{V}\text{d}\bar{z}, 
 \end{equation}
 where \begin{equation}\label{eq:pq}
     \mathcal{U} = \begin{pmatrix}
 0 & 0 & p_{1} & p_{2} \\
 0 & 0 & p_{3} & p_{4} \\
 -p_{1} & -p_{3} & 0 & q \\
 -p_{2} & -p_{4} & -q & 0
\end{pmatrix}, \quad \mathcal{V} = \begin{pmatrix}
 0 & 0 & \bar{p}_{1} & \bar{p}_{2} \\
 0 & 0 & \bar{p}_{3} & \bar{p}_{4} \\
 -\bar{p}_{1} & -\bar{p}_{3} & 0 & \bar{q} \\
 -\bar{p}_{2} & -\bar{p}_{4} & -\bar{q} & 0
\end{pmatrix}, 
 \end{equation}
 \begin{equation*}
	\left\{
		\begin{aligned}
		&p_{1} = \frac{1}{\sqrt{2}}\left( \la e_{3z} , \mathfrak{f} \ra + \la e_{3z} , \bar{\mathfrak{f}} \ra \right) = -\frac{1}{\sqrt{2}} \left( \frac{\alpha + e^{u} + \bar{\beta}}{\sqrt{2e^{u} + \alpha + \bar{\alpha}}} \right), \\
		&p_{2} = \frac{1}{\sqrt{2}}\left( \la e_{4z} , \mathfrak{f} \ra + \la e_{4z} , \bar{\mathfrak{f}}  \ra \right) = \frac{i}{\sqrt{2}} \left[ \frac{\left( \alpha\bar{\alpha} - e^{2u} \right) - \bar{\beta}\left( e^{u} + \alpha \right)}{\sqrt{\left( 2e^{u} + \alpha + \bar{\alpha} \right)\left( e^{2u} - \alpha\bar{\alpha} \right) }} \right], \\
        &p_{3} = -\frac{i}{\sqrt{2}}\left( \la e_{3z}, \mathfrak{f}  \ra - \la e_{3z}, \bar{\mathfrak{f}} \ra \right) = \frac{i}{\sqrt{2}} \left( \frac{\alpha + e^{u} - \bar{\beta}}{\sqrt{2e^{u} + \alpha + \bar{\alpha}}} \right),\\
        &p_{4} = -\frac{i}{\sqrt{2}}\left( \la e_{4z}, \mathfrak{f} \ra - \la e_{4z}, \bar{\mathfrak{f}} \ra \right) = \frac{1}{\sqrt{2}} \left[ \frac{\left( \alpha\bar{\alpha} - e^{2u} \right) + \bar{\beta}\left( e^{u} + \alpha \right)}{\sqrt{\left( 2e^{u} + \alpha + \bar{\alpha} \right)\left( e^{2u} - \alpha\bar{\alpha} \right) }} \right], \\
        &q = \la e_{4z}, e_{3} \ra = i \left[ \frac{\frac{1}{2}e^{u}\left( \alpha_{z} - \bar{\alpha}_{z} \right) + \frac{1}{2}\left( \alpha_{z}\bar{\alpha} - \bar{\alpha}_{z}\alpha \right) - e^{u}\phi\left( 2e^{u} + \bar{\alpha} + \alpha \right) - u_{z}e^{u}\left( e^{u} + \alpha \right)}{\left( 2e^{u} + \alpha + \bar{\alpha} \right)\sqrt{e^{2u} - \alpha\bar{\alpha}}} \right], 
		\end{aligned}
		\right.	
\end{equation*}
  and 
  \begin{equation}\label{complex function}
\beta := \la \mathfrak{f}_{z}, \mathfrak{f}_{\bar{z}} \ra, \quad \phi := e^{-u} \la \mathfrak{f}_{z\bar{z}}, \overline{\mathfrak{f}_{\bar{z}}}\ra.
  \end{equation}
From the above expressions, it is easy to see that 
 $\alpha \, \text{d}z^2$, $\beta  \, \text{d}z \text{d}\bar z$ and $\phi \, \text{d}z$
 are well-defined differentials for the horizontal lift $\mathfrak f$. Moreover 
 it is known that 
 \[
 \Phi= \phi \, \text{d}z
 \]
 characterizes minimality of the surface, i.e., $\Phi =0$ if and only if the surface 
 $f$ is minimal \cite[Theorem 3.2]{WX}.

It is clear that the compatibility condition of $\mathcal F$, that is, $\mathcal{F}_{z\bar{z}} = \mathcal{F}_{\bar{z}z}$ is equivalent to $\mathcal{U}_{\bar{z}} - \mathcal{V}_{z} = \left[ \mathcal{U}, \mathcal{V} \right]$, which is also equivalent to the Maurer-Cartan equation $\text{d}\omega + \frac{1}{2}\left[ \omega \wedge \omega \right] = 0$. Moreover, it can also be interpreted as a flatness of the connection $\text{d} + \omega$. 
Then a straightforward computation shows that the Maurer-Cartan equation $\text{d}\omega + \frac{1}{2}\left[ \omega \wedge \omega \right] = 0$ is equivalent to the following equations:
 \begin{subnumcases}{}
	&$p_{1}\bar{p}_{3} + p_{2}\bar{p}_{4} = \bar{p}_{1}p_{3} + \bar{p}_{2}p_{4}$, \label{eq:M-C1}\\
	&$\bar{p}_{2}q - p_{2}\bar{q} = p_{1\bar{z}} - \bar{p}_{1z}$, \label{eq:M-C2}\\
	&$p_{1}\bar{q} - \bar{p}_{1}q = p_{2\bar{z}} - \bar{p}_{2z}$, \label{eq:M-C3}\\
        &$\bar{p}_{4}q - p_{4}\bar{q} = p_{3\bar{z}} - \bar{p}_{3z}$, \label{eq:M-C4}\\
        &$p_{3}\bar{q} - \bar{p}_{3}q = p_{4\bar{z}} - \bar{p}_{4z}$, \label{eq:M-C5}\\
        &$p_{2}\bar{p}_{1} + p_{4}\bar{p}_{3} - \bar{p}_{2}p_{1} - \bar{p}_{4}p_{3} = q_{\bar{z}} - \bar{q}_{z}$.  \label{eq:M-C6}
\end{subnumcases}
By substituting $p_{1}, p_{2}, p_{3}, p_{4}$ and $q$, \eqref{eq:M-C1} can be simplified as follows:
\begin{equation}\label{eq:alphabeta}
        e^{2u} - \beta\bar{\beta} = \alpha \bar{\alpha}.
    \end{equation}
From the above equation, the assumption $e^{2u} - \alpha \bar{\alpha} >  0$ is equivalent to $|\beta| > 0$.
And \eqref{eq:M-C2} and \eqref{eq:M-C4} are equivalent to the following equation: 
    \begin{equation*}
        \frac{\sqrt{2}}{2}\left( \bar{p}_{2} - i\bar{p}_{4} \right)q + \frac{\sqrt{2}}{2}\left( \bar{p}_{1} - i\bar{p}_{3} \right)_{z} = \frac{\sqrt{2}}{2}\left( p_{1} - ip_{3} \right)_{\bar{z}} + \frac{\sqrt{2}}{2}\left( p_{2} - ip_{4} \right)\bar{q}.
    \end{equation*}
    By substituting $p_{1}, p_{2}, p_{3}, p_{4}$ and $q$, we have
    \begin{equation}
        \frac{1}{2}\bar{\alpha}_{\bar{z}}\alpha - e^{2u}\bar{\phi} - u_{\bar{z}}e^{2u} + \beta \bar{\beta}_{\bar{z}} = \frac{1}{2} \bar{\alpha}_{z}\beta. \label{eq: condition 5}
    \end{equation}
    Similarly, \eqref{eq:M-C3} and \eqref{eq:M-C5} are equivalent to the following equation:
    \begin{equation*}
        \frac{\sqrt{2}}{2}\left( p_{1} - ip_{3} \right)\bar{q} - \frac{\sqrt{2}}{2}\left( p_{2} - ip_{4} \right)_{\bar{z}} = \frac{\sqrt{2}}{2}\left( \bar{p}_{1} - i\bar{p}_{3} \right)q - \frac{\sqrt{2}}{2}\left( \bar{p}_{2} - i\bar{p}_{4} \right)_{z}.
    \end{equation*}
    By substituting $p_{1}, p_{2}, p_{3}, p_{4}$ and $q$, we obtain 
   \begin{equation}
        \phi \beta = \bar{\phi} \alpha + \frac{1}{2} \alpha_{\bar{z}}. \label{eq: condition 6}
    \end{equation}
 Note that \eqref{eq:M-C6} can also be expressed by the function $\alpha$, $\beta$ and 
 $\phi$, however, the expression is rather involved and will not be used later, 
 so we omit it here.
 
\subsection{Minimality of a Lagrangian surface and the family of flat connections}\label{sbsc: Minimality}
Since $Q_2$ is a symmetric space as in \eqref{eq:symQ2}
and minimal surfaces in $Q_2$
can be thought as conformal harmonic maps, thus the integrable system method can be applied.
We consider the following family of connection 1-forms $\text{d} + \omega^{\lambda}$:
 \begin{equation}\label{lambda family}
     \omega^{\lambda} = \lambda^{-1}\omega^{'}_{\mathfrak{p}} + \omega_{\mathfrak{k}} + \lambda\omega^{''}_{\mathfrak{p}}, \quad 
     (\lambda \in \mathbb{S}^{1}),
 \end{equation}
where $\mathfrak{g} = \operatorname{Lie}\left( \rm{SO}(4) \right)= \mathfrak{so}(4)$ can be decomposed as 
\[
\mathfrak{g} = \mathfrak{k} \oplus \mathfrak{p}
\]
with the fixed point subalgebra $\mathfrak{k} = \operatorname{Fix} (d \sigma)=  \mathfrak{so}(2) \times \mathfrak{so}(2)$ and its complement $\mathfrak p$, and $\omega_{\mathfrak{k}}$ and $\omega_{\mathfrak{p}}$ are the $\mathfrak{k}$- and the $\mathfrak{p}$-valued 1-forms.
Moreover $\prime$ and $\prime \prime$ denote the $(1,0)$- and the $(0,1)$-parts, respectively. Since the flatness of the connection $\text{d} + \omega$ is only the flatness of the family $\text{d} + \omega^{\lambda}$ at $\lambda = 1$, the flatness of a family $d + \omega^{\lambda}$ parameterized by $\lambda \in \mathbb{S}^1$ gives 
 an additional constraint, that is, harmonicity, for a Lagrangian surface $f$. 

More explicitly, \eqref{lambda family} can be written as follows:
\begin{equation*}
    \omega^{\lambda} = \mathcal{U}^{\lambda}\text{d}z + \mathcal{V}^{\lambda}\text{d}\bar{z},
\end{equation*}
where \begin{equation*}
     \mathcal{U}^{\lambda} = \begin{pmatrix}
 0 & 0 & \lambda^{-1}p_{1} & \lambda^{-1}p_{2} \\
 0 & 0 & \lambda^{-1}p_{3} & \lambda^{-1}p_{4} \\
 -\lambda^{-1}p_{1} & -\lambda^{-1}p_{3} & 0 & q \\
 -\lambda^{-1}p_{2} & -\lambda^{-1}p_{4} & -q & 0
\end{pmatrix}, \quad \mathcal{V}^{\lambda} = \begin{pmatrix}
 0 & 0 & \lambda\bar{p}_{1} & \lambda\bar{p}_{2} \\
 0 & 0 & \lambda\bar{p}_{3} & \lambda\bar{p}_{4} \\
 -\lambda\bar{p}_{1} & -\lambda\bar{p}_{3} & 0 & \bar{q} \\
 -\lambda\bar{p}_{2} & -\lambda\bar{p}_{4} & -\bar{q} & 0
\end{pmatrix}.
 \end{equation*}
The condition $\text{d}\omega^{\lambda} + \frac{1}{2}[ \omega^{\lambda} \wedge \omega^{\lambda} ] = 0$ is now equivalent to $\mathcal{U}^{\lambda}_{\bar{z}} - \mathcal{V}^{\lambda}_{z} = [ \mathcal{U}^{\lambda} , \mathcal{V}^{\lambda} ]$, and it is equivalent to the following equations:
\begin{subnumcases}{}
	&$p_{1}\bar{p}_{3} + p_{2}\bar{p}_{4} = \bar{p}_{1}p_{3} + \bar{p}_{2}p_{4}$, \label{eq:l M-C1}\\
	&$\lambda\bar{p}_{2}q-\lambda^{-1}p_{2}\bar{q} = \lambda^{-1}p_{1\bar{z}} - \lambda\bar{p}_{1z}$, \label{eq:l M-C2}\\
	&$\lambda^{-1}p_{1}\bar{q} - \lambda\bar{p}_{1}q = \lambda^{-1}p_{2\bar{z}} -\lambda\bar{p}_{2z}$, \label{eq:l M-C3}\\
        &$\lambda\bar{p}_{4}q - \lambda^{-1}p_{4}\bar{q} = \lambda^{-1}p_{3\bar{z}} - 
    \lambda\bar{p}_{3z}$, \label{eq:l M-C4}\\
        &$\lambda^{-1}p_{3}\bar{q} - \lambda\bar{p}_{3}q = \lambda^{-1}p_{4\bar{z}} -\lambda\bar{p}_{4z}$, \label{eq:l M-C5}\\
        &$p_{2}\bar{p}_{1} + p_{4}\bar{p}_{3} - \bar{p}_{2}p_{1} - \bar{p}_{4}p_{3} = q_{\bar{z}} - \bar{q}_{z}$.  \label{eq:l M-C6}
\end{subnumcases}
First, \eqref{eq:l M-C1} is just \eqref{eq:M-C1} and it is 
$e^{2u} - \beta\bar{\beta} = \alpha \bar{\alpha}$. 
    Next, \eqref{eq:l M-C2} and \eqref{eq:l M-C4} are equivalent to the following equations:
 \begin{equation*}
\frac{\sqrt{2}}{2}\left( \bar{p}_{2} - i\bar{p}_{4} \right)q = -\frac{\sqrt{2}}{2}\left( \bar{p}_{1} - i\bar{p}_{3} \right)_{z}, \quad 
\frac{\sqrt{2}}{2}\left( \bar{p}_{2} + i\bar{p}_{4} \right)q = -\frac{\sqrt{2}}{2}\left( \bar{p}_{1} + i\bar{p}_{3} \right)_{z}. 
\end{equation*}
    By substituting $p_{1}, p_{2}, p_{3}, p_{4}$ and $q$, the following equations are obtained:
 \begin{align}
		&\frac{1}{2}\bar{\alpha}_{z}= e^{u}\phi, \label{eq: condition 1}\\
        &\frac{1}{2}\bar{\alpha}\left( \bar{\alpha}_{z} - \alpha_{z} \right) + e^{u}\bar{\alpha}_{z} + u_{z}e^{2u} = \bar{\beta} \beta_{z} \label{eq: condition 2}. 
  \end{align} 
    Similarly, \eqref{eq:l M-C3} and \eqref{eq:l M-C5} are equivalent to the following equations:
 \begin{equation*}
\frac{\sqrt{2}}{2}\left( \bar{p}_{1} + i\bar{p}_{3} \right)q = \frac{\sqrt{2}}{2}\left( \bar{p}_{2} + i\bar{p}_{4} \right)_{z}, \quad
\frac{\sqrt{2}}{2}\left( \bar{p}_{1} - i\bar{p}_{3} \right)q = \frac{\sqrt{2}}{2}\left( \bar{p}_{2} - i\bar{p}_{4} \right)_{z}. 
\end{equation*}
    By substituting $p_{1}, p_{2}, p_{3}, p_{4}$ and $q$, we have $\bar{\alpha}_{z} = 0$. Thus \eqref{eq: condition 1} and \eqref{eq: condition 2} can be simplified as follows:
     \begin{align}
		&\alpha_{\bar{z}} = \bar{\alpha}_{z} = \phi = 0, \label{eq: condition 3} \\
        &\frac{1}{2}\bar{\alpha}\alpha_{z} + \bar{\beta}\beta_{z} - u_{z}e^{2u} = 0. \label{eq: condition 4}
  \end{align} 
In addition, the equation \eqref{eq:l M-C6} can be simplified as follows:
\begin{equation}\label{eq: second order PDE}
    u_{z\bar{z}}e^{u}|\beta|^{2} - \frac{1}{4}|\alpha_{z}|^{2} e^{u} - |u_{z}|^{2} e^{u} |\alpha|^{2}  + \frac{1}{2} \alpha_{z} u_{\bar{z}} e^{u} \bar{\alpha} + \frac{1}{2}\bar{\alpha}_{\bar{z}} u_{z} e^{u} \alpha + 2|\beta|^{4} = 0.
\end{equation}
Thus \eqref{eq:l M-C1}-\eqref{eq:l M-C6} can be simplified to the following system of equations:
 \begin{align}\label{eq: M-C}
	\left\{
		\begin{aligned}
		& e^{2u} - \beta\bar{\beta} = \alpha \bar{\alpha}, \\
		&\alpha_{\bar{z}} = \bar{\alpha}_{z} = \phi = 0, \\
        &\frac{1}{2}\bar{\alpha}\alpha_{z} + \bar{\beta}\beta_{z} - u_{z}e^{2u} = 0, \\
        & u_{z\bar{z}}e^{u}|\beta|^{2} - \frac{1}{4}|\alpha_{z}|^{2} e^{u} - |u_{z}|^{2} e^{u} |\alpha|^{2}  + \frac{1}{2} \alpha_{z} u_{\bar{z}} e^{u} \bar{\alpha} + \frac{1}{2}\bar{\alpha}_{\bar{z}} u_{z} e^{u} \alpha + 2|\beta|^{4} = 0.
		\end{aligned}
		\right.	
\end{align}

Without loss of generality, we assume that $\alpha$ and $\beta$ are not identically zero,
by the following Lemma: 
\begin{Lemma}\label{lem:totally}
When $\alpha \equiv 0$ (resp. $\beta \equiv 0$), then the minimal Lagrangian surface is 
the totally geodesic Lagrangian sphere (resp. totally geodesic flat Lagrangian torus).
\end{Lemma}
\begin{proof}
When $\alpha \equiv 0$, the equation \eqref{eq: second order PDE} can be simplified to 
\begin{equation}\label{eq: Gauss curvature of alpha}
    u_{z\bar{z}} + 2e^{u} = 0,
\end{equation}
which is known as Liouville's equation,  see \cite{L1838}.
Since the induced metric on $\mathbb{D}$ is given by $ds^{2}_{M} = 2e^{u}\text{d}z\text{d}\bar{z}$, the Gaussian curvature $K$ can be computed by $K = -e^{-u}u_{z\bar{z}}$.
Then by \eqref{eq: Gauss curvature of alpha}, the Gaussian curvature $K = 2$ when $\alpha \equiv 0$.

On the contrary, when $\beta \equiv 0$, we have $\alpha(z) \ne 0$ by $e^{2u} = \alpha\bar{\alpha}$. Notice that $\alpha_{z} = 2u_{z}\alpha$, so we have
$\alpha_{z\bar{z}} = 2u_{z\bar{z}}\alpha + 2u_{z}\alpha_{\bar{z}} = 2u_{z\bar{z}}\alpha = 0$.
Thus the Gaussian curvature $K = 0$ when $\beta \equiv 0$. 

 It is well known that $Q_{2}$ is isometric to the product space $\mathbb{S}^{2} \times \mathbb{S}^{2}$, and  
 using Theorem 4.3 (1)  \cite{CU} for minimal Lagrangian surfaces in $\mathbb{S}^{2} \times \mathbb{S}^{2}$, then $f(M)$ is congruent to some open subset of the totally geodesic Lagrangian sphere when $\alpha \equiv 0$ or totally geodesic flat Lagrangian torus when $\beta \equiv 0$, see \cite{CU} in detail.
\end{proof}
 
\begin{Remark}
 The last PDE of \eqref{eq: M-C} is an involved second-order equation with respect to $u$, however, 
 introducing a function $\hat u$ such that $2 e^u = e^{\hat u } + |\alpha|^2 e^{-\hat u}$ holds, it becomes 
 the elliptic sinh-Gordon equation, as we will see in Proposition \ref{prp:corrsinh}.
 \end{Remark}

We now characterize minimal Lagrangian surfaces in terms of the family of flat
connections $\text{d} + \omega^{\lambda}$.
\begin{Theorem}\label{Thm1}
    Let $f : M \to Q_{2}$ be a Lagrangian immersion and let $\text{d} + \omega^{\lambda}$ be the family of connections in \eqref{lambda family}. Then the following statements are equivalent:  
    \begin{enumerate}
        \item The Lagrangian immersion $f$ is minimal.
        \item The connections $\text{d} + \omega^{\lambda}$ are flat for all $\lambda \in \mathbb{S}^{1}$, i.e. \eqref{eq:l M-C1}-\eqref{eq:l M-C6} holds for all $\lambda \in \mathbb{S}^{1}$.
        \item The quadratic differential $\alpha \, \mathrm{d} z^2$ is holomorphic and $\varphi = \arg{(\beta)}$ is constant, where $\alpha, \beta$ are defined in \eqref{eq:alpha}, \eqref{complex function} and $\arg{(\beta)}$ denotes the argument of $\beta$.
    \end{enumerate}
\end{Theorem}
\begin{proof}
    We would like to prove this theorem by (1) $\Rightarrow$ (2) $\Rightarrow$ (3) $\Rightarrow$ (1).
    
    (1) $\Rightarrow$ (2): 
    We assume $\Phi = \phi\, \text{d}z = 0$. Thus the equation \eqref{eq: condition 6} can be simplified to \eqref{eq: condition 3} and the equation \eqref{eq: condition 5} can be simplified to \eqref{eq: condition 4}.
    Thus \eqref{eq:l M-C1}-\eqref{eq:l M-C6} holds for all $\lambda \in \mathbb{S}^{1}$.

    (2) $\Rightarrow$ (3): 
    Note that $\alpha_{\bar{z}} = 0$ in \eqref{eq: M-C}.
    By $e^{2u} - \beta \bar{\beta} = \alpha \bar{\alpha}$ and $\alpha_{\bar{z}} = 0$, we have
    \begin{equation}\label{eq: alpha}
         2 u_{\bar{z}}e^{2u} = \alpha \bar{\alpha}_{\bar{z}} +  \beta_{\bar{z}} \bar{\beta} + \beta\bar{\beta}_{\bar{z}}.      
    \end{equation}
    Thus the third equation of \eqref{eq: M-C}, $\frac{1}{2}\bar{\alpha}\alpha_{z} + \bar{\beta}\beta_{z} - u_{z}e^{2u} = 0$, can be simplified to
    \begin{equation}\label{eq: beta}
        \beta \bar{\beta}_{\bar{z}} = \beta_{\bar{z}}\bar{\beta}.
    \end{equation}
    And if we put $\beta = |\beta|e^{i\varphi}$, then \eqref{eq: beta} can be simplified to $2i\varphi_{\bar{z}}|\beta|^{2} = 0$. Since $\beta$ is not identically zero, we obtain $\varphi_{\bar{z}} = 0$.
    Since $\varphi$ takes values in $\mathbb{R}$, it follows that  $\varphi$ is constant. 
    
    (3) $\Rightarrow$ (1): 
    By $\alpha_{\bar{z}} = 0$, we still have \eqref{eq: alpha}. 
    Since $\varphi$ is constant, it is easy to check $\bar{\beta}_{\bar{z}}\beta = \bar{\beta}\beta_{\bar{z}} = |\beta||\beta|_{\bar{z}}$, and then \eqref{eq: alpha} becomes 
    \begin{equation}
        u_{\bar{z}}e^{2u} = \frac{1}{2}\alpha\bar{\alpha}_{\bar{z}} + |\beta||\beta|_{\bar{z}}.
    \end{equation}
     If we substitute this into the equation \eqref{eq: condition 5}, then we obtain $\phi = 0$. 
\end{proof}
We now consider a new local horizontal lift of a minimal Lagrangian surface $f$:
\begin{equation}\label{eq:newlift}
\hat{\mathfrak{f}} = e^{-\frac{i\varphi}{2}}\mathfrak{f},
\end{equation}
 where $\varphi(z, \bar{z}) = \arg(\beta)$ is constant with respect to $z$ and $\bar{z}$ by Theorem \ref{Thm1}. It is easy to check that $\hat{\mathfrak{f}}$ satisfies the horizontal condition,
and the new complex functions $\hat{\alpha}$ and $\hat{\beta}$ can be given by
\begin{equation}
    \hat{\alpha} := \la  \hat{\mathfrak{f}}_{z}, \hat{\mathfrak{f}}_{z}  \ra = e^{-i\varphi}\alpha, \quad \hat{\beta} := \la \hat{\mathfrak{f}}_{z},\hat{\mathfrak{f}}_{\bar{z}} \ra = |\beta|,
\end{equation}
i.e. $|\hat{\alpha}| = |\alpha|$, and $ \hat{\beta}$ is a non-negative real function. 
By the relation $e^{2u} - |\alpha|^{2} = |\beta|^{2}$, we obtain that 
\begin{equation}\label{eq:hatbeta}
\hat{\beta} = \sqrt{e^{2u} - |\hat{\alpha}|^{2}},
\end{equation}
 i.e., the function $\hat{\beta}$ is completely determined by $u$ and $\hat \alpha$.
Note that $\hat \alpha$ is a holomorphic function for a minimal Lagrangian surface.
All the data in Maurer-Cartan form now can be represented by $u$ and $\hat \alpha$,
and the new complex functions $\hat{p}_{1}, \hat{p}_{2}, \hat{p}_{3}, \hat{p}_{4}, \hat{q}$ 
are introduced as follows:
\begin{equation*}
		\hat{p}_{1} := -\frac{1}{\sqrt{2}}\left( \frac{\hat{\alpha} + e^{u} + \hat{\beta} }{\sqrt{2e^{u} + \hat{\alpha} + \bar{\hat{\alpha}} }} \right), \,
		\hat{p}_{2} := -\frac{i}{\sqrt{2}}\left( \frac{\hat{\alpha} + e^{u}  + \hat{\beta}}{\sqrt{2e^{u} + \hat{\alpha} + \bar{\hat{\alpha}} }} \right), \,
        \hat{p}_{3} := \frac{i}{\sqrt{2}}\left( \frac{\hat{\alpha} + e^{u} - \hat{\beta} }{\sqrt{2e^{u} + \hat{\alpha} + \bar{\hat{\alpha}} }} \right),
\end{equation*}
\begin{equation*}     
        \hat{p}_{4} := \frac{1}{\sqrt{2}}\left( \frac{\hat{\alpha} + e^{u} - \hat{\beta} }{\sqrt{2e^{u} + \hat{\alpha} + \bar{\hat{\alpha}} }} \right), \quad    \hat{q} :=  i \left( \frac{e^{-u}\frac{1}{2}\hat{\alpha}_{z}\hat{\beta} - e^{-u} \hat{\alpha} \hat{\beta}_{z} - \hat{\beta}_{z}}{2e^{u} + \hat{\alpha} + \bar{\hat{\alpha}} } \right).
\end{equation*}
\begin{Remark}\label{Rm:assumption}
Since the factor $e^{2u}- \alpha \bar \alpha$ in the denominator of 
$p_1, p_2, p_3, p_4$ and $q$ can be eliminated 
 by the same factor in the numerator, thus 
 $\hat{p}_{1}, \hat{p}_{2}, \hat{p}_{3},\hat{p}_{4}$ and $\hat{q}$ are well-defined even when 
 $e^{2u}- \alpha \bar \alpha = 0$. Therefore, we do not need to assume the condition 
 $e^{2u}- \alpha \bar \alpha > 0$ for any minimal Lagrangian surface.
\end{Remark}
\subsection{Minimal Lagrangian surfaces and the elliptic sinh-Gordon equation}
 To simplify the Maurer-Cartan form, we now introduce complex functions $p$ and $r$ as follows:
\begin{equation}\label{eq:pr}
    p:= \frac{1}{\sqrt{2}}\left( \frac{\hat{\alpha}+e^{u}-\hat{\beta}}{\sqrt{2e^{u} + \hat{\alpha} + \bar{\hat{\alpha}} }} \right), \quad r:= -\frac{1}{\sqrt{2}} \left( \frac{\hat{\alpha} + e^{u} + \hat{\beta} }{\sqrt{2e^{u} + \hat{\alpha} + \bar{\hat{\alpha}} }}\right).
\end{equation}
Namely, we have 
\begin{equation}
    \hat{p}_{1} = r,\quad \hat{p}_{2} = ir,\quad \hat{p}_{3} = ip,\quad \hat{p}_{4} = p.
\end{equation}

\begin{Definition}\label{def:extend}
 Denote the new frame of the lift $\hat{\mathfrak f}$ in \eqref{eq:newlift}
  of a minimal Lagrangian immersion by $\hat{\mathcal F}$. 
 Then by Theorem \ref{Thm1}, there exists a family of 
 frames $\hat{\mathcal{F}}_{\lambda}$ such that $\hat{\mathcal{F}}_{\lambda}|_{\lambda=1} = \hat{\mathcal F}$, and we call $\hat{\mathcal F}_{\lambda}$
 the \textit{extended frame}.
\end{Definition}
The new family of connection 1-forms $\text{d} + \hat{\omega}^{\lambda}$ parameterized by $\lambda \in \mathbb{S}^{1}$ can be explicitly written as follows:
\begin{equation}\label{eq: M-C hat omega}
    \hat{\omega}^{\lambda} = \hat{\mathcal{F}}_{\lambda}^{-1}\text{d}\hat{\mathcal{F}}_{\lambda} = \hat{\mathcal{U}}^{\lambda}\text{d}z + \hat{\mathcal{V}}^{\lambda}\text{d}\bar{z},
\end{equation}
where
\begin{equation*}
     \hat{\mathcal{U}}^{\lambda} = \begin{pmatrix}
 0 & 0 & \lambda^{-1}r & \lambda^{-1}ir \\
 0 & 0 & \lambda^{-1}ip & \lambda^{-1}p \\
 -\lambda^{-1}r & -\lambda^{-1}ip & 0 & \hat{q} \\
 -\lambda^{-1}ir & -\lambda^{-1}p & -\hat{q} & 0
\end{pmatrix}, \quad \hat{\mathcal{V}}^{\lambda} = \begin{pmatrix}
 0 & 0 & \lambda\bar{r} & -\lambda i\bar{r} \\
 0 & 0 & -\lambda i\bar{p} & \lambda\bar{p} \\
 -\lambda\bar{r} & \lambda i\bar{p} & 0 & \bar{\hat{q}} \\
 \lambda i\bar{r} & -\lambda\bar{p} & -\bar{\hat{q}} & 0
\end{pmatrix}.
 \end{equation*}
Then the flatness condition leads to 
\begin{align}\label{eq: M-C pr}
p_{\bar{z}} = ip\bar{\hat{q}}, \quad r_{\bar{z}} = -ir\bar{\hat{q}} \quad \mbox{and}\quad
        -i\frac{\hat{q}_{\bar{z}}}{2} + i\frac{\bar{\hat{q}}_{z}}{2} = |r|^{2} - |p|^{2}. 
\end{align}
By the first and second equations of \eqref{eq: M-C pr}, we obtain that
\begin{equation}\label{eq: iq}
i\bar{\hat{q}} = \left( \log p \right)_{\bar{z}} = -\left( \log r \right)_{\bar{z}}.
\end{equation}
On the other hand, note that 
\begin{equation}\label{eq: pr0}
    |p|^{2} = \frac{1}{2}\left( e^{u} - \hat{\beta} \right), \quad |r|^{2} = \frac{1}{2}\left( e^{u} + \hat{\beta} \right)
\end{equation}
hold. These imply that
\begin{equation}\label{eq: pr}
    |r|^{2} - |p|^{2} = \hat{\beta}, \quad |p|^{2}|r|^{2} = \frac{1}{4}|\hat{\alpha}|^{2}.
\end{equation} 
Thus by the third equation of \eqref{eq: M-C pr}, \eqref{eq: iq} and \eqref{eq: pr}, we have
$-\left( \log |r|^{2} \right)_{z\bar{z}} = 2\hat{\beta}$,
which can be written as follows:
 \begin{equation}\label{eq: M-C r beta2}
       -\frac{1}{2}\left( \log |r|^{2} \right)_{z\bar{z}} - \hat{\beta} = -\frac{1}{2}\left( \log |r|^{2} \right)_{z\bar{z}} + \frac{1}{4}|\hat{\alpha}|^{2}|r|^{-2} - |r|^{2} = 0.
\end{equation}
By choosing a real function $\hat{u}$ by 
$\hat{u} := \log (2 |r|^2)$, 
\eqref{eq: M-C r beta2} leads to
\begin{equation}\label{eq: sinh-Gordon equation}
    \hat{u}_{z\bar{z}} + e^{\hat{u}} - |\hat{\alpha}|^{2}e^{-\hat{u}} = 0,
\end{equation}
that is the \textit{elliptic sinh-Gordon equation.}
Note that by \eqref{eq: pr0}, we can also represent the metric on $\mathbb{D} \subset M$ by 
\begin{equation}\label{eq: metric}
    2 e^u \text{d}z\text{d}\bar{z} = (e^{\hat{u}} + |\hat{\alpha}|^{2}e^{-\hat{u}})\text{d}z\text{d}\bar{z}. 
\end{equation}
By this correspondence, we have established the following proposition.
\begin{Proposition}\label{prp:corrsinh}
 Let $u$ be the conformal factor of the metric of a minimal Lagrangian surface $f$, and let $\hat u$   
 be a function defined by \eqref{eq: metric}. Then $u$ satisfies the last PDE of \eqref{eq: M-C}
 if and only if $\hat u$ satisfies \eqref{eq: sinh-Gordon equation}.\
\end{Proposition}

\subsection{A family of minimal Lagrangian surfaces}\label{sbsc:family}
To obtain a family of minimal Lagrangian surfaces of a given minimal Lagrangian surface $f$, 
let us consider the \textit{gauge transformation} of the extended frame $\hat{\mathcal{F}}_{\lambda}$. 
For a family of smooth maps $\mathcal{G}_{\lambda}: \mathbb{D} \to \rm{SO}(2) \times \rm{SO}(2)$ 
parametrized by $\lambda \in \mathbb S^1$, let $\tilde{\mathcal{F}}_{\lambda} := \hat{\mathcal{F}}_{\lambda}\mathcal{G}_{\lambda}$.
 Then a straightforward computation shows that 
\begin{equation*}
    \tilde{{\omega}}^{\lambda} := \tilde{\mathcal{F}}_{\lambda}^{-1}\text{d}\tilde{\mathcal{F}}_{\lambda} = \mathcal{G}_{\lambda}^{-1}\hat{\mathcal{F}}_{\lambda}^{-1}(\text{d}\hat{\mathcal{F}}_{\lambda}\mathcal{G}_{\lambda} + \hat{\mathcal{F}}_{\lambda}\text{d}\mathcal{G}_{\lambda}) = \mathcal{G}_{\lambda}^{-1}\hat{\omega}^{\lambda}\mathcal{G}_{\lambda} + \mathcal{G}_{\lambda}^{-1}\text{d}\mathcal{G}_{\lambda}
\end{equation*}
holds. We explicitly define $\mathcal{G}_{\lambda}$ as follows:
For $p$ and $r$ in \eqref{eq:pr}, it is easy to check $pr = -\frac{1}{2}\hat{\alpha}$, $i.e.$ $p = -\frac{1}{2}\hat{\alpha}r^{-1}$. Thus by setting $r = |r|e^{i\arg r}$, 
where $\arg r \in \mathbb{R}$ is the argument of $r$, then $p = -\frac{1}{2}\hat{\alpha}|r|^{-1}
e^{-i\arg r}$ holds. Moreover, let 
$\lambda = e^{i \theta} (\theta \in \R)$ and we define a gauge $\mathcal{G}_{\lambda}$  by 
\begin{equation*}
    \mathcal{G}_{\lambda} = \begin{pmatrix}
 1 & 0 & 0 & 0 \\
 0 & 1 & 0 & 0 \\
 0 & 0 & \cos (\arg r- \theta) & \sin (\arg r- \theta) \\
 0 & 0 & -\sin (\arg r -\theta) & \cos (\arg r- \theta)
\end{pmatrix} \in \rm{SO}(2) \times \rm{SO}(2).
\end{equation*}
Since $|r| = \frac{1}{\sqrt{2}}e^{\frac{\hat{u}}{2}}$ holds, we obtain that
\begin{equation}\label{eq: Maurer-Cartan form tilde}
    \tilde{\omega}^{\lambda} = \tilde{\mathcal{U}}^{\lambda}\text{d}z + \tilde{\mathcal{V}}^{\lambda}\text{d}\bar{z},
\end{equation}
where
\begin{equation}\label{eq: tilde U}
    \tilde{\mathcal{U}}^{\lambda} = \frac{\sqrt{2}}{2}\begin{pmatrix}
 0 & 0 & e^{\hat{u}/2} & ie^{\hat{u}/2} \\
 0 & 0 & -i \lambda^{-2}\hat{\alpha}e^{-\hat{u}/2} & -\lambda^{-2}\hat{\alpha}e^{-\hat{u}/2} \\
 -e^{\hat{u}/2} & i\lambda^{-2}\hat{\alpha}e^{-\hat{u}/2} & 0 & -\frac{i}{\sqrt{2}}\hat{u}_{z} \\
 -ie^{\hat{u}/2} & \lambda^{-2}\hat{\alpha}e^{-\hat{u}/2} & \frac{i}{\sqrt{2}}\hat{u}_{z} & 0
\end{pmatrix}, 
\end{equation}
and
\begin{equation}\label{eq: tilde V}
\tilde{\mathcal{V}}^{\lambda} = \frac{\sqrt{2}}{2}\begin{pmatrix}
 0 & 0 & e^{\hat{u}/2} & -ie^{\hat{u}/2} \\
 0 & 0 & i \lambda^2 \bar{\hat{\alpha}}e^{-\hat{u}/2} & -\lambda^2 \bar{\hat{\alpha}}e^{-\hat{u}/2} \\
 - e^{\hat{u}/2} & - i\lambda^2 \bar{\hat{\alpha}}e^{-\hat{u}/2} & 0 & \frac{i}{\sqrt{2}}\hat{u}_{\bar{z}} \\
 ie^{\hat{u}/2} & \lambda^2 \bar{\hat{\alpha}}e^{-\hat{u}/2} & -\frac{i}{\sqrt{2}}\hat{u}_{\bar{z}} & 0
\end{pmatrix}.
\end{equation}

From the form of $\tilde \omega^{\lambda}$, it is easy to see that there exists 
a family of minimal Lagrangian surfaces $f^{\lambda}$ such that the following fundamental 
quantities
\[
\hat u^{\lambda} = \hat u, \quad \hat \alpha^{\lambda} = \lambda^{-2} \hat \alpha, \quad 
\hat \beta^{\lambda} = \hat \beta.
\]
Note that the metric of $f^{\lambda}$ is given by 
\[
       (e^{\hat{u}} + |\hat{\alpha}^{\lambda}|^{2}e^{-\hat{u}})\text{d}z\text{d}\bar{z} =  (e^{\hat{u}} + |\hat{\alpha}|^{2}e^{-\hat{u}})\text{d}z\text{d}\bar{z},
\] 
and the surface $f^{\lambda}$ is obtained by the sum of the first two columns for 
$\tilde{\mathcal F}_{\lambda}$, i.e., 
\[
f^{\lambda}=[\hat{\mathfrak f}^{\lambda}] \in Q_2, \quad 
\hat{\mathfrak f}^{\lambda} = \frac{\sqrt{2}}2(\tilde{\mathcal F}_{\lambda}^1 + i \tilde{\mathcal F}_{\lambda}^2) \in \C^4,
\]
where $\tilde{\mathcal F}_{\lambda}^j (j=1, 2)$ denotes the $j$-th column of 
$\tilde{\mathcal F}_{\lambda}$.
We summarize the above discussion as the following theorem:
\begin{Theorem}\label{thm2}
   If $f: M \to Q_{2}$ is a minimal Lagrangian immersion with the induced metric $2 e^{u}\mathrm{d}z 
   \mathrm{d}\bar z$ and the  holomorphic quadratic differential $\hat \alpha\,  \mathrm{d}z^{2}$, then there exists a family of minimal Lagrangian immersions $\{ f^{\lambda} \}$ parameterized by $\lambda \in \mathbb{S}^{1}$, which has the same induced metric $2 e^{u}\mathrm{d}z 
   \mathrm{d}\bar z$ and the holomorphic quadratic differential $\hat \alpha^{\lambda} \mathrm{d}z^{2} = \lambda^{-2} \hat \alpha \, \mathrm{d}z^{2}$.
\end{Theorem}
 \subsection{Minimal Lagrangian surfaces and minimal surfaces in $\mathbb S^3$}
 We now compare the Maurer-Cartan form of 
a minimal Lagrangian surface with that of a minimal surface in $\mathbb{S}^{3}$. Then we establish the following theorem.
\begin{Theorem}\label{Thm corresponding}
 Any minimal surface $f_{min}$ in $\mathbb S^3$ with unit normal $N$, the metric 
 $2 e^{\hat u}\, \mathrm{d} z\mathrm{d}\bar{z}$ and 
 the Hopf differential $\mathcal{Q}\, \mathrm{d}z^{2}$, 
 gives a minimal Lagrangian surface  $f=[f_{min} + iN ] \in Q_2$ with 
 the holomorphic differential $-i \mathcal{Q} \mathrm{d}z^{2}$ and the metric 
 \begin{equation}\label{eq:liftmetric}
 2 e^{u}\, \mathrm{d} z\mathrm{d}\bar{z}
   = \bigl(e^{\hat u} + |\mathcal{Q}|^{2} e^{-\hat u}\bigr)\, \mathrm{d} z\mathrm{d}\bar{z}.
\end{equation}
 
 Conversely, for any minimal Lagrangian surface $f$  in $Q_{2}$ with the holomorphic differential 
 $\hat \alpha \, \mathrm{d}z^2$ and the metric 
 $2e^{u}\text{d}z\text{d}\bar{z}$, there exists a unique map $g= (f_{min}, N)$ 
 which takes values of the unit tangent bundle of $\mathrm{U}\mathbb S^3 = \mathbb S^3 \times \mathbb S^2$ with the metric of $g$ is $4e^{u}\text{d}z\text{d}\bar{z}$. Furthermore, 
 both projections $f_{min}$ and $N$ share the Hopf differential $i\hat \alpha  \, \mathrm{d}z^2$
 and the metrics $2e^{\hat{u}}\text{d}z\text{d}\bar{z}$ of $f_{min}$ 
 and $2e^{\tilde{u}}\text{d}z\text{d}\bar{z}$ of $N$
 are given by 
 \begin{equation*}
  e^{\hat{u}} = e^{u} + \sqrt{e^{2u} - |\hat{\alpha}|^{2}}, \quad e^{\tilde{u}} = e^{u} - \sqrt{e^{2u} - |\hat{\alpha}|^{2}}. 
 \end{equation*}
\end{Theorem}

\begin{proof}
 Let $f_{min}$ be a minimal surface in $\mathbb{S}^{3}$ and $\mathbb{D} \subset M$ be a simply connected domain with conformal  coordinate $z$. Note that the induced metric on $\mathbb{D}$ can be written as $ds^{2}_{M} = 2e^{\hat{u}}\text{d}z\text{d}\bar{z}$. 
 Now let $\mathcal{Q} \mathrm{d}z^{2} = i\hat{\alpha} \mathrm{d}z^{2}$ be the Hopf differential, that is the 
 $(2,0)$-part of the second fundamental form and it is holomorphic for $f_{min}$, 
 and  choose the $\SO$-frame of $f_{min}$ as follows:
 \begin{equation}\label{eq: frame sigma}
     \sigma = \left( f_{min}, N , -\frac{(f_{min})_{z} + (f_{min})_{\bar{z}}}{\sqrt{2}e^{\hat{u}/2}}, \frac{i\left( (f_{min})_{z} - (f_{min})_{\bar{z}} \right)}{\sqrt{2}e^{\hat{u}/2}} \right), 
 \end{equation}
where $N$ is the normal vector of $f_{min}$.
By $\mathcal{Q} \mathrm{d}z^{2} = i\hat{\alpha} \mathrm{d}z^{2}$ and the results in \cite{Bobenko1991, SKKR}, the Maurer-Cartan form obtained by \eqref{eq: frame sigma} is the same as \eqref{eq: Maurer-Cartan form tilde} for $\lambda = 1$. It follows that the minimal Lagrangian surface $f$ and its local lift $\hat{\mathfrak{f}}$ defined above can be obtained by 
\begin{equation}
f=[ \hat{\mathfrak{f}}] \in Q_{2}, \quad \hat{\mathfrak{f}} := \frac{\sqrt{2}}{2}\left( f_{min} + iN \right) \in \mathbb{C}^{4}.
\end{equation}
By direct computation, we have
\begin{equation*}
    \langle \hat{\mathfrak{f}}_{z} , \hat{\mathfrak{f}}_{z} \rangle = 
    \frac12 \left\langle  (f_{min})_{z} + iN_{z}, (f_{min})_{z} + iN_{z} \right\rangle = -i\mathcal{Q} = \hat{\alpha}.
\end{equation*}
 Thus the first claim follows.
 
 Note that, the normal vector $N$ can be viewed as the dual minimal surface of $f_{min}$. After choosing the $\SO$-frame of $N$ properly as follows:
\begin{equation}\label{eq: frame sigma1}
     \tilde{\sigma} = \left( N, f_{min} , -\frac{N_{z} + N_{\bar{z}}}{\sqrt{2}e^{\tilde{u}/2}}, \frac{i\left( N_{z} - N_{\bar{z}} \right)}{\sqrt{2}e^{\tilde{u}/2}} \right), 
 \end{equation}
the minimal Lagrangian surface $\tilde{f}$ can be obtained by
\begin{equation}
\tilde{f}=[ N + if_{min}] = [f_{min} - iN] = \big[ \; \bar{\hat{\mathfrak{f}}} \; \big] \in Q_{2} .
\end{equation}
It shows that $\{ f_{min} , N \}$ corresponds to the same minimal Lagrangian surface $f$ in $Q_{2}$ up to a conjugation. 

Conversely, let $f$ be a minimal Lagrangian surface in $Q_{2}$ with the holomorphic differential $\hat{\alpha} \mathrm{d}z^{2}$ and $\tilde{\omega}^{\lambda}$ be the Maurer-Cartan form defined in \eqref{eq: Maurer-Cartan form tilde}. Since $\tilde{\omega}^{\lambda}|_{\lambda = 1}$ is the same as the Maurer-Cartan form of the $\SO$-frame \eqref{eq: frame sigma} of $f_{min}$, there exists a pair of minimal surfaces in $\mathbb{S}^{3}$ as follows:
\begin{equation}
    f_{min}= \frac{1}{\sqrt{2}}\big( \hat{\mathfrak{f}} + \bar{\hat{\mathfrak{f}}} \big) = \sqrt{2} 
    \operatorname{Re}  \hat{\mathfrak{f}} , \quad N = -\frac{i}{\sqrt{2}}\big( \hat{\mathfrak{f}} - \bar{\hat{\mathfrak{f}}} \big) = \sqrt{2}\operatorname{Im}\hat{\mathfrak{f}},
\end{equation}
where $\hat{\mathfrak{f}}$ is the local lift defined above. 
On the other hand, since the Maurer-Cartan form $\tilde{\omega}^{\lambda}|_{\lambda = 1}$ could be also identified by another Maurer-Cartan form of the $\SO$-frame \eqref{eq: frame sigma1} of $N$, there exists another pair of minimal surfaces in $\mathbb{S}^{3}$ as follows:
\begin{equation}
    N = \sqrt{2} \operatorname{Re}\hat{\mathfrak{f}}, \quad f_{min} = \sqrt{2}\operatorname{Im}\hat{\mathfrak{f}}.
\end{equation}
In conclusion, there exists a unique pair of minimal surfaces $\{( f_{min}, 2e^{\hat{u}}\text{d}z\text{d}\bar{z}) , (N, 2e^{\tilde{u}}\text{d}z \text{d}\bar{z})\}$ in $\mathbb{S}^3$. 
By similar computation, we have
\begin{equation*}
    \mathcal{Q} = \langle (f_{min})_{zz}, N  \rangle = -\left\langle (f_{min})_{z} , N_{z} \right\rangle = \left\langle N_{zz}, f_{min}  \right\rangle = i\hat{\alpha}.
\end{equation*}
This completes the proof.
\end{proof}

\begin{Remark}
\mbox{}
\begin{enumerate}
  \item The metric defined in \eqref{eq: metric} can be thought as half the Sasakian metric 
 on the unit tangent bundle $\mathrm{U}\mathbb S^3 = \mathbb S^3 \times \mathbb S^2$ for 
 a minimal surface $f_{min}$ in $\mathbb S^3$, that is, the following relation holds:
 \begin{equation*}
2 ds^2_{Lag}= 4 e^{u}\, \mathrm{d} z\mathrm{d}\bar{z}
   =  \bigl(2 e^{\hat u} + 2|\mathcal{Q}|^{2} e^{-\hat u}\bigr)\, \mathrm{d} z\mathrm{d}\bar{z}
   = ds^2_{Sasaki}.
\end{equation*}
  \item  The converse statement in Theorem \ref{Thm corresponding} has been shown in \cite[Theorem 4.4]{CU} 
 under the assumption $|\hat \alpha| > 0$, through the PDE point view. 
 More generally, such correspondences of hypersurfaces in unit spheres and 
Lagrangian immersions in Grassmannians have been studied in \cite{P1997, VW2020}.
 \end{enumerate}

\end{Remark}
In order to highlight how minimal surfaces in $\mathbb{S}^3$ can be constructed from minimal Lagrangian surfaces in $Q_2$, we state the following corollary.
\begin{Corollary}\label{Cor fmin N}
Let $f: M \to Q_{2}$ be a minimal Lagrangian surface and $\hat{\mathfrak{f}}$ be its local lift defined above. Then the minimal surface $f_{min}$ and its dual minimal surface $N$ in $\mathbb{S}^{3}$ associated with $f$ are of the form
\begin{equation}
    f_{min} = \sqrt{2}\operatorname{Re}\hat{\mathfrak{f}}, \quad N = \sqrt{2} \operatorname{Im} \hat{\mathfrak{f}}.
\end{equation}
\end{Corollary}

\section{The loop group method for minimal Lagrangian surfaces}\label{sc:DPW}
 The PDE in \eqref{eq: sinh-Gordon equation} is merely
 the structure equation for a harmonic map from a Riemann surface into $\mathbb S^2$.
 Thus, it is natural to expect the existence of such a map associated with any minimal Lagrangian surface.
 Moreover, we will establish the loop group method for minimal Lagrangian surfaces following the approach in \cite{DPW}.

\subsection{Minimal Lagrangian surfaces and harmonic maps into $\mathbb S^2$}
\label{subsc:S2harm}
To obtain an associated harmonic map to $\mathbb S^2$, we 
use the following Lie group isomorphism \cite{H2024}
\[
 \SO = (\SU \times \SU)/\boldsymbol{Z}_2,
\]
see Appendix \ref{app: lie group lie algebra} for details.

For the Maurer-Cartan form $\omega$, we can use Lie algebra isomorphism $\mathfrak{so}(4) \cong \mathfrak{su
}(2) \oplus \mathfrak{su}(2)$, see also in Appendix \ref{app: lie group lie algebra}. Accordingly, we can also use $\Lambda \mathfrak{so}(4) \cong \Lambda \mathfrak{su}(2) \oplus \Lambda \mathfrak{su}(2)$, where 
\[
\Lambda \mathfrak{g} := \{ g : \mathbb S^1 \to \mathfrak g \}
\]
 denotes the \textit{loop algebra} of $\mathfrak g$, which is an infinite-dimensional Banach Lie
 algebra with respect to a suitable norm \cite{PS}.
 Hence the associated Maurer-Cartan form $\hat{\omega}^{\lambda}$ in \eqref{eq: M-C hat omega} can 
  be represented as follows:
\begin{equation}
\begin{aligned}\label{eq:hatu}
\hat{U}^{\lambda}= (\hat{U}_1^{\lambda}, \hat{U}_2^{\lambda})=&\left( \begin{pmatrix}
 \frac{1}{4}\hat{u}_{z} & -\frac{\sqrt{2}}{2}\lambda^{-1} e^{\hat{u}/2} \\
  \frac{\sqrt{2}}{2}\lambda^{-1}\hat{\alpha}e^{-\hat{u}/2} & -\frac{1}{4}\hat{u}_{z}
\end{pmatrix}, \quad \begin{pmatrix}
 \frac{1}{4}\hat{u}_{z} & -i\frac{\sqrt{2}}{2}\lambda^{-1} e^{\hat{u}/2} \\
  i\frac{\sqrt{2}}{2}\lambda^{-1}\hat{\alpha} e^{-\hat{u}/2} & -\frac{1}{4}\hat{u}_{z}
\end{pmatrix} \right), 
\end{aligned}
\end{equation}
and
\begin{equation}\label{eq:hatv}
\begin{aligned}
\hat{V}^{\lambda} =(\hat{V}_1^{\lambda}, \hat{V}_2^{\lambda})= &\left( \begin{pmatrix}
-\frac{1}{4}\hat{u}_{\bar{z}} & -\frac{\sqrt{2}}{2}\lambda \bar{\hat{\alpha}}e^{-\hat{u}/2} \\
  \frac{\sqrt{2}}{2} \lambda e^{\hat{u}/2} & \frac{1}{4}\hat{u}_{\bar{z}}
\end{pmatrix}, \quad \begin{pmatrix}
 -\frac{1}{4}\hat{u}_{\bar{z}} &i\frac{\sqrt{2}}{2}\lambda \bar{\hat{\alpha}}e^{-\hat{u}/2} \\
 -i\frac{\sqrt{2}}{2}\lambda e^{\hat{u}/2} & \frac{1}{4}\hat{u}_{\bar{z}}
\end{pmatrix} \right).
\end{aligned}
\end{equation}
 From the form of $\left( \hat{U}_{1}^{\lambda}, \hat{U}_{2}^{\lambda} \right)$ and $\left( \hat{V}_{1}^{\lambda}, \hat{V}_{2}^{\lambda} \right)$,  the corresponding pair of maps is 
 given by $(F_{\lambda}, F_{-i\lambda})$ such that 
\begin{equation}\label{eq: Maurer-Cartan form tilde12}
\hat{\omega}^{\lambda}_{1} 
= F_{\lambda}^{-1}\text{d}F_{\lambda} = \hat{U}_{1}^{\lambda} \text{d}z + \hat{V}_{1}^{\lambda} \text{d}\bar{z}, \quad 
\hat{\omega}^{\lambda}_{2} 
= F_{-i\lambda}^{-1}\text{d}F_{-i\lambda} = \hat{U}_{2}^{\lambda} \text{d}z + \hat{V}_{2}^{\lambda} \text{d}\bar{z}.
\end{equation}
Notice that $\hat{U}_{2}^{\lambda}$ (resp. $\hat{V}_{2}^{\lambda}$) can be obtained from $\hat{U}_{1}^{\lambda}$ (resp. $\hat{V}_{1}^{\lambda}$) by transforming $\lambda$ into $-i\lambda$.
For further simplification, we can choose $w := \sqrt{2} z$ as our new coordinate and $ds^{2}_{M} = 2e^{\check{u}} \text{d} w \text{d} \bar{w}$ as our new metric.
And ${Q} \text{d}w^{2} := \hat{\alpha} \text{d}z^{2}$ is our holomorphic quadratic differential under the new coordinate $w$.
In addition, the associated holomorphic quadratic differential $\lambda^{-2}\hat{\alpha}\text{d}z^{2}$ mentioned in Theorem \ref{thm2} becomes $\lambda^{-2}Q\text{d}w^{2}$.
From the above discussion, we obtain the following proposition.
\begin{Proposition}\label{prp:extendcorr}
 Let $\hat {\mathcal F}_{\lambda}$ be the extended frame of $\hat{\mathfrak f}$. Then there exists the pair of maps $(F_{\lambda}, F_{-i \lambda})$ such that 
 the Maurer-Cartan form $\hat{\omega}^{\lambda}_{1}$ of  $F_{\lambda}$ is given by \eqref{eq: Maurer-Cartan form tilde12}.  Conversely, 
 given a pair of maps $(F_{\lambda}, F_{-i \lambda})$ with the Maurer-Cartan form $\hat{\omega}^{\lambda}_{1}$ of $F_{\lambda}$ is given by \eqref{eq: Maurer-Cartan form tilde12}, there exists the extended frame $\hat {\mathcal F}_{\lambda}$ 
 of some minimal surface $\hat{\mathfrak f}$.
\end{Proposition}

Now we consider the map $\check{f}: M \to \mathbb{S}^{2}$ and focus on the associated Maurer-Cartan form $\check{\omega}^{\lambda}$ induced by the local lift $F_{\lambda}$, where $\check{\omega}^{\lambda}$ can be written as follows:
\begin{equation}\label{eq: system}
    \check{\omega}^{\lambda} = F_{\lambda}^{-1}\text{d}F_{\lambda} = \check{U}^{\lambda} \text{d}w + \check{V}^{\lambda}\text{d}\bar{w},
\end{equation}
where
\begin{equation}
    \check{U}^{\lambda} = \begin{pmatrix}
 \frac{1}{4}\check{u}_{w} & -\frac{1}{2}\lambda^{-1}e^{\check{u}/2}  \\
  \lambda^{-1}Qe^{-\check{u}/2} & -\frac{1}{4}\check{u}_{w}
\end{pmatrix}, \quad \check{V}^{\lambda} = \begin{pmatrix}
-\frac{1}{4}\check{u}_{\bar{w}} & -\lambda \bar{Q}e^{-\check{u}/2} \\
  \frac{1}{2}\lambda e^{\check{u}/2} & \frac{1}{4}\check{u}_{\bar{w}}
\end{pmatrix}.
\end{equation}
The Maurer-Cartan form $\check \omega^{\lambda}$ is the same to the Maurer-Cartan form 
of the extended frame of a non-conformal harmonic map from a Riemann surface $M$ into 
$\mathbb S^2$. Thus the map $F_{\lambda}$ is called the \textit{extended frame} 
of a harmonic map $\check{f}$. 
In addition, let $F_{\lambda}$, $F_{-i\lambda}$ be two solutions of the system \eqref{eq: system} and define 
\begin{equation}\label{eq: fmin N}
    f_{min} := F_{\lambda} \begin{pmatrix}
 e^{\pi i/4} & 0 \\
 0 & e^{-\pi i/4}
\end{pmatrix} F_{-i\lambda}^{-1}, \quad N:= iF_{\lambda}\begin{pmatrix}
 e^{\pi i/4} & 0 \\
 0 & -e^{-\pi i/4}
\end{pmatrix}F_{-i\lambda}^{-1}.
\end{equation}
By the results in \cite{Bobenko1991, FKR2009}, $f_{min}$ is a minimal surface in $\mathbb{S}^{3}$ and normal $N$.

\subsection{The DPW method }\label{Subsection the DPW method}\label{sbsc:DPW}
In Section \ref{subsc:S2harm}, we have obtained 
the family of Maurer-Cartan forms $\check \omega^{\lambda}$ for 
a harmonic into $\mathbb S^2$ from a minimal Lagrangian surface $f$ in $Q_2$, 
thus we can apply the generalized Weierstrass type representation (the DPW method)
for harmonic maps in $\mathbb S^2$. The basic construction can be found in 
\cite{DH2003, FKR2009, SKKR}. First note that the extended frame $F_{\lambda}$ takes values in the 
loop group of $G = \SU$
\[
\Lambda G_{\sigma} = \{g : \mathbb S^1  \to G \mid  \sigma g (\lambda) = g(-\lambda) \},
\]
where $\sigma$ is the involution $\sigma (g) = \operatorname{Ad} \operatorname{diag}(1, -1)  (g)$ is the involution of the symmetric space $\mathbb S^2 
= \SU/\Uone$. By introducing a suitable topology, $\Lambda G_{\sigma}$ 
becomes a Banach Lie group and we call it the \textit{loop group} of 
$G$, see \cite{PS}. The extended frame $F_{\lambda}$ clearly takes values in the loop group of $\SU$, and the \textit{Birkhoff decomposition} (\cite{PS, DPW}) of $F_{\lambda}$
\[
 F_{\lambda} = F_{-} F_{+}\quad \mbox{with $F_- \in \Lambda_*^- G^{\mathbb C}_{\sigma}$
 and $F_+\in \Lambda^+ G^{\mathbb C}_{\sigma}$}
\]
yields the meromorphic dependence of $F_{-}$ \cite[Lemma 2.6]{DPW}. Here 
$G^{\mathbb C} = \SL$ and $\Lambda^{\pm} G^{\mathbb C}_{\sigma}$ 
denote the subgroups of $\Lambda G^{\mathbb C}_{\sigma}$ which can be 
extended to the inside (resp. outside) of unit disk in $\mathbb{C} P^1$ 
when the plus sign (resp. the minus sign) is chosen. The subscript $*$ denotes 
the identity normalization at $\lambda =0$.

Conversely, minimal Lagrangian surfaces in the complex quadric $Q_{2}$ can be constructed in the following four steps:
\begin{enumerate}
    \item[\textbf{Step 1:}] Solve the initial-value problem: 
    \begin{equation}
        d\Phi = \Phi \xi, \quad \Phi(z_{0}) = \Phi_{0} \in \Lambda \SL_{\sigma},
    \end{equation}
    to obtain a unique map $\Phi : \mathbb D  \to \Lambda \SL_{\sigma}$.
    \item[\textbf{Step 2:}] Compute the \textit{Iwasawa decomposition} (see \cite{PS, DPW})  
    of $\Phi$ pointwise on $\mathbb D$:
    \begin{equation}\label{eq: F lambda}
        \Phi = F_{\lambda} B, \quad F_{\lambda} \in \Lambda \SU_{\sigma}, \quad B \in \Lambda^{+} \SL_{\sigma},
    \end{equation}
    Then by \cite[Lemma 4.2]{DPW}, $F_{\lambda}$ is the extended frame of a harmonic map into 
    $\mathbb S^2$. Set the pair of maps given by another map $F_{-i\lambda}$ as
    \begin{equation}
        \left( F_{\lambda}, F_{-i\lambda} \right) \in \Lambda \SU_{\sigma}\times \Lambda \SU_{\sigma}.
    \end{equation}
    
    \item[\textbf{Step 3:}] Use the Loop group isomorphism 
    \begin{equation}\label{eq:loopisom}
        \Lambda \SO_{\sigma} \cong (\Lambda \SU_{\sigma} \times \Lambda \SU_{\sigma}) /\boldsymbol{Z}_2, 
    \end{equation}
    and Proposition \ref{prp:extendcorr}, one obtains 
    the extended frame $\mathcal{F}_{\lambda} \in \Lambda \SO_{\sigma} $ of 
    some minimal Lagrangian immersion into $Q_2$. 
 \item[\textbf{Step 4:}] Finally, by using the formula \eqref{eq:formula} in the 
 following proposition, we obtain 
 a family of minimal Lagrangian immersions $f^{\lambda}$ into $Q_2$.
\end{enumerate}
\begin{Remark}
The involution $\sigma$ of the twisted loop group $\Lambda \SO_{\sigma}$ is
   associated with the Hermitian symmetric space $ Q_2 = \SO/ \rm{SO}(2) \times \rm{SO}(2)$.
   It is explicitly given by $\sigma = \operatorname{Ad} \operatorname{diag} (1, 1, -1, -1)$ on $\SO$,
   and it is  different from the involution $\operatorname{Ad} \operatorname{diag} (1, -1)$ of $\SU$
   which is associated with the twisted loop group $\Lambda \SU_{\sigma}$.
   However, the homomorphism $\psi: \SU \times \SU \to \SO$ in \eqref{eq:psi} can be 
   naturally extended to the twisted loop group level as \eqref{eq:loopisom}, see the corresponding Maurer-Cartan form in 
   \eqref{eq:hatu} and \eqref{eq:hatv}.
\end{Remark}
\begin{Proposition}\label{Prop}
Let $F_{\lambda}$ be the extended frame defined above.
 Set 
 \begin{equation}\label{eq:Xlambda}
     X^{\lambda} = (X^{\lambda}_{ij}) := F_{\lambda}F_{-i\lambda}^{-1} , \quad Y^{\lambda} = (Y^{\lambda}_{ij}) := i F_{\lambda} \sigma_{3} F_{-i\lambda}^{-1}, \quad 
     \mbox{with} \quad \sigma_{3} = \begin{pmatrix} 1 & 0 \\
 0 & -1
\end{pmatrix}.
 \end{equation}
 Then the associated family $\{f^{\lambda} \}$ of $f$ can be represented by 
     \begin{equation}\label{eq:formula}
     f^{\lambda}=
\left[\left(
\operatorname{R} (X^{\lambda}_{11}) + i\operatorname{R} (Y^{\lambda}_{11}),  
\operatorname{I} (X^{\lambda}_{11}) + i\operatorname{I} (Y^{\lambda}_{11}), 
\operatorname{R} (X^{\lambda}_{12}) + i\operatorname{R} (Y^{\lambda}_{12}), 
\operatorname{I} (X^{\lambda}_{12}) + i\operatorname{I} (Y^{\lambda}_{12})
\right)
\right],
\end{equation}
where $\operatorname{R}$ and $\operatorname{I}$ denotes the real and the imaginary part, respectively.
\end{Proposition}
\begin{proof}
By the Lie group isomorphism given in Appendix \ref{app: lie group lie algebra} and a straightforward computation, we obtain the result.
\end{proof}

\begin{Corollary}\label{coro:formula}
Let $F_{\lambda}$ be the extended frame defined above.
Define a map 
\begin{equation}
\Phi_{\lambda}= 
(\phi_{\lambda}, \psi_{\lambda})
: \mathbb D \to  \mathbb S^2 \times \mathbb S^2
\end{equation}
by $(\phi_{\lambda}, \psi_{\lambda})
:= ( i F_{\lambda} \sigma_3 F_{\lambda}^{-1}, i F_{-i\lambda} \sigma_3 F_{-i\lambda}^{-1})$.
Then  $\{\Phi_{\lambda}\}_{\lambda \in \mathbb S^1}$ 
is a family of minimal Lagrangian surfaces.
\end{Corollary}
\begin{proof}
By the definition of the maps $\phi_{\lambda}$ and $\psi_{\lambda}$, 
it is easy to check $\phi_{\lambda}, \psi_{\lambda} \in \mathfrak{su}(2) \cong \mathbb{R}^{3}$ for each $\lambda \in \mathbb{S}^{1}$ and we have
\begin{equation*}
    \la \phi_{\lambda}, \phi_{\lambda}  \ra = \frac{1}{2} \mathrm{tr}( \phi_{\lambda} \sigma_{2} \phi_{\lambda}^{T} \sigma_{2} )= 1, \quad \sigma_{2} = \begin{pmatrix}
 0 & -i \\
 i & 0
\end{pmatrix}, 
\end{equation*}
because of $F_{\lambda}, F_{-i\lambda} \in \SU$ for each $\lambda \in \mathbb{S}^{1}$,  
 and similarly for $|\psi_{\lambda}|^2=1$,  i.e. $\phi_{\lambda}, \psi_{\lambda} \in \mathbb{S}^{2}$. Moreover
\begin{equation*}
\begin{aligned}
    (\phi_{\lambda})_{z} &= iF_{\lambda}(\hat{U}^{\lambda}_{1}\sigma_{3} - \sigma_{3}\hat{U}^{\lambda}_{1})F_{\lambda}^{-1} = F_{\lambda} \begin{pmatrix}
 0 & i\sqrt{2}\lambda^{-1}e^{\hat{u}/2} \\
i\sqrt{2}\lambda^{-1}\hat{\alpha}e^{-\hat{u}/2} & 0
\end{pmatrix}F_{\lambda}^{-1}, \\
    (\phi_{\lambda})_{\bar{z}} &= iF_{\lambda}(\hat{V}^{\lambda}_{1}\sigma_{3} - \sigma_{3}\hat{V}^{\lambda}_{1})F_{\lambda}^{-1} = F_{\lambda} \begin{pmatrix}
 0 & i\sqrt{2}\lambda \bar{\hat{\alpha}}e^{-\hat{u}/2} \\
 i\sqrt{2}\lambda e^{\hat{u}/2} & 0
\end{pmatrix}F_{\lambda}^{-1}, 
\end{aligned} 
\end{equation*}
and similarly for $(\psi_{\lambda})_{z}$ and $(\psi_{\lambda})_{\bar z}$.
Then a straightforward computation shows that 
\begin{equation*}
\begin{aligned}
   \la (\phi_{\lambda})_{\bar z},  (\phi_{\lambda})_{z} \ra &= 
 \la (\psi_{\lambda})_{\bar z},  (\psi_{\lambda})_{z} \ra = 2 e^{u}, \\
  \la (\Phi_{\lambda})_{\bar z}, (\Phi_{\lambda})_{z}\ra   &=\left \langle (\phi_{\lambda})_{\bar{z}}, (\phi_{\lambda})_{z} \right \rangle + \left \langle (\psi_{\lambda})_{\bar{z}}, (\psi_{\lambda})_{z} \right \rangle = 4e^{u},\\
        \left\langle (\Phi_{\lambda})_{z}, (\Phi_{\lambda})_{z}  \right \rangle &=  \frac{1}{2}\mathrm{tr} \left( (\phi_{\lambda})_{z} \sigma_{2} (\phi_{\lambda})_{z}^{T} \sigma_{2} \right) + \frac{1}{2}\mathrm{tr} \left( (\psi_{\lambda})_{z} \sigma_{2} (\psi_{\lambda})_{z}^{T} \sigma_{2} \right) = 0
\end{aligned}    
\end{equation*} 
hold. That is, $\Phi_{\lambda}$ is a conformal immersion, and moreover,  it is Lagrangian.
Then by  another straightforward computation shows that 
\begin{equation*}
\begin{aligned}
    (\phi_{\lambda})_{z\bar{z}} &= iF_{\lambda}\left( \left(\hat{V}_{1}^{\lambda}\hat{U}_{1}^{\lambda} + (\hat{U}_{1}^{\lambda})_{\bar{z}} \right)\sigma_{3} - \hat{U}_{1}^{\lambda}\sigma_{3}\hat{V}_{1}^{\lambda} - \hat{V}^{\lambda}_{1}\sigma_{3}\hat{U}_{1}^{\lambda} + \sigma_{3}\left( -(\hat{U}^{\lambda}_{1})_{\bar{z}} + \hat{U}^{\lambda}_{1}\hat{V}^{\lambda}_{1}  \right)  \right)F_{\lambda}^{-1} \\
                                &= iF_{\lambda}\left( -2e^{u}\sigma_{3} \right)F_{\lambda}^{-1} =-2e^{u} \phi_{\lambda}, 
 \end{aligned}
 \end{equation*}
 holds and similarly $(\psi_{\lambda})_{z\bar{z}} = -2e^{u} \psi_{\lambda}$ follows.
Thus $(\Phi_{\lambda})_{z \bar z} = -2e^{u} \Phi_{\lambda}$ holds, and $\Phi_{\lambda}$ is harmonic, and 
it is minimal.\footnote{It can be also seen by \eqref{eq: CU 1} in Appendix \ref{app: CU}.}
\end{proof}

\section{Examples of minimal Lagrangian surfaces through the 
DPW method}\label{sc:Ex}
By using the DPW method introduced in Section \ref{sc:DPW}, we will construct minimal Lagrangian surfaces 
in $Q_2$. 
\subsection{Basic examples}
The basic examples are the totally geodesic spheres and the totally geodesic tori.
\subsubsection{Totally geodesic sphere}
Define
\begin{equation*}
    \xi = \lambda^{-1} \begin{pmatrix}
 0 & 1\\
 0 & 0
\end{pmatrix}\text{d}z
\end{equation*}
for $z \in \mathbb{C}$.
By using this holomorphic potential, we can construct a family of totally geodesic spheres in $Q_{2}$ as follows, see also \cite{CU}: 

It is easy to solve the ODE $d \Phi = \Phi \xi$ by $\Phi = \exp ( z \xi/ \text{d}z)$ with $\Phi (0)=\id$.
 Moreover,  the Iwasawa decomposition of $\Phi = F_{\lambda} B$ is given by 
\begin{equation*}
    F_{\lambda} = \frac{1}{\sqrt{1 + |z|^{2}}}\begin{pmatrix}
 1 & z\lambda^{-1} \\
 -\bar{z}\lambda & 1
\end{pmatrix}. 
\end{equation*}
By the loop group isomorphism $\Lambda \SO_{\sigma} \cong (\Lambda \SU_{\sigma} \times \Lambda \SU_{\sigma}) /\boldsymbol{Z}_2$, we have
\begin{equation*}
    \mathcal{F}_{\lambda} =  \frac{1}{1 + |z|^{2}} \begin{pmatrix}
 1 & -|z|^{2} & \ast & \ast \\
 -|z|^{2} & 1 & \ast & \ast \\
 \frac{1}{2}\left( z\lambda^{-1} + \bar{z}\lambda - i\left( z\lambda^{-1} -\bar{z}\lambda \right) \right) & \frac{1}{2}\left( z\lambda^{-1} + \bar{z}\lambda - i\left( z\lambda^{-1} -\bar{z}\lambda \right) \right) & \ast & \ast \\
 -\frac{1}{2}\left( z\lambda^{-1} + \bar{z}\lambda + i\left( z\lambda^{-1} -\bar{z}\lambda \right) \right) & -\frac{1}{2}\left( z\lambda^{-1} + \bar{z}\lambda + i\left( z\lambda^{-1} -\bar{z}\lambda \right) \right) & \ast & \ast
\end{pmatrix}.
\end{equation*}
Finally, by the first and second columns of $\mathcal{F}_{\lambda}$, we can obtain a family of totally geodesic spheres $\{ f^{\lambda} \}$ parameterized by $\lambda \in \mathbb{S}^{1}$
\begin{equation*}
f^{\lambda}= 
\left[
    \begin{pmatrix}
1 - i|z|^{2} \\
-|z|^{2} + i \\
z\lambda^{-1} + i\bar{z}\lambda \\
-\bar{z}\lambda - iz\lambda^{-1}
\end{pmatrix}
\right].
\end{equation*}
\subsubsection{Totally geodesic torus}
Define
\begin{equation*}
    \xi = \lambda^{-1} \begin{pmatrix}
 0 & 1\\
 1 & 0
\end{pmatrix}\text{d}z
\end{equation*}
for $z \in \mathbb{C}$.
By using this holomorphic potential, we can construct a family of totally geodesic torus in $Q_{2}$
as follows, see also in \cite{CU}:

 It is easy to solve the ODE $d \Phi = \Phi \xi$ by $\Phi = \exp ( z \xi/ \text{d}z)$ with $\Phi (0)=\id$.
 Moreover,  the Iwasawa decomposition of $\Phi = F_{\lambda} B$ is given by 
\begin{equation*}
    F_{\lambda} = \begin{pmatrix}
 \cosh (\lambda^{-1}z - \bar{z}\lambda) & \sinh (\lambda^{-1}z - \bar{z}\lambda)\\
 \sinh (\lambda^{-1}z - \bar{z}\lambda) & \cosh (\lambda^{-1}z - \bar{z}\lambda)
\end{pmatrix}.
\end{equation*}
By the loop group isomorphism $\Lambda \SO_{\sigma} \cong (\Lambda \SU_{\sigma} \times \Lambda \SU_{\sigma}) /\boldsymbol{Z}_2$, we have
\begin{equation*}
    \mathcal{F}_{\lambda} = \begin{pmatrix}
 \cos \left( \lambda^{-1}z + \bar{z}\lambda + i(\lambda^{-1}z - \bar{z}\lambda) \right) & 0 & \ast & \ast \\
 0 & \cos \left( \lambda^{-1}z + \bar{z}\lambda - i(\lambda^{-1}z - \bar{z}\lambda) \right) & \ast & \ast \\
 0 & \sin \left( \lambda^{-1}z + \bar{z}\lambda - i(\lambda^{-1}z - \bar{z}\lambda) \right) & \ast & \ast \\
 -\sin \left( \lambda^{-1}z + \bar{z}\lambda + i(\lambda^{-1}z - \bar{z}\lambda) \right) & 0 & \ast & \ast
\end{pmatrix}
\end{equation*}
Finally, by the first and second columns of $\mathcal{F}_{\lambda}$, we obtain a family of totally geodesic tori $\{ f^{\lambda} \}$ parameterized by $\lambda \in \mathbb{S}^{1}$
\begin{equation}
f^{\lambda}=   \left[
\begin{pmatrix}
\cos \left( \lambda^{-1}z + \bar{z}\lambda + i(\lambda^{-1}z - \bar{z}\lambda) \right)  \\
i\cos \left( \lambda^{-1}z + \bar{z}\lambda - i(\lambda^{-1}z - \bar{z}\lambda) \right) \\

i\sin \left( \lambda^{-1}z + \bar{z}\lambda - i(\lambda^{-1}z - \bar{z}\lambda) \right) \\
-\sin \left( \lambda^{-1}z + \bar{z}\lambda + i(\lambda^{-1}z - \bar{z}\lambda) \right)
\end{pmatrix}\right].
\end{equation}

\subsection{Equivariant and radially symmetric examples}
 We now show two new examples of minimal Lagrangian surfaces in $Q_2$.
\subsubsection{$\mathbb R$-equivariant minimal Lagrangian surfaces}\label{subsub:equiv}
\begin{Definition}[$\mathbb R$-equivariant potentials, \cite{SKKR}]
 Define 
 \begin{equation}\label{eq: Delaunay potential}
    \xi = D(\lambda) \frac{\mathrm{d} z}{z}, \quad \text{where} \quad D(\lambda) = \begin{pmatrix}
 c & a\lambda^{-1} + b\lambda \\
 a\lambda + b\lambda^{-1} & -c
\end{pmatrix}, 
\end{equation}
with $a, b, c \in \R$ for $z$ in $\mathbb{D} = \mathbb{C} \setminus \{ 0 \}$. 
We call such potentials the \textit{Equivariant potentials}.
\end{Definition}
 It is easy to see that $\Phi = \exp (\log z \cdot D)$ is the unique solution of $d \Phi = \Phi \xi$ 
 with the initial condition  $\Phi|_{z=0} = \id$.
 Let 
 \[
 \Phi = F_{\lambda} B 
 \]
 be the Iwasawa decomposition of $\Phi$ (see below for the explicit form of $F_{\lambda}$ for $c=0$). 
 As discussed in \cite[Section 2.5]{FKR2009}, by the rotation of the domain $z \to e^{i\theta} \cdot z$,
 the following transformation rule of $F_{\lambda}$ follows: 
\begin{equation}\label{eq: equivariant frame}
    F_{\lambda}\left( e^{i\theta} \cdot z , \, e^{-i\theta} \cdot \bar{z} , \, \lambda \right) = \exp \left( i\theta D (\lambda) \right) \cdot F_{\lambda}\left( z ,\, \bar{z} ,\, \lambda \right). 
\end{equation}
 Note that $i \theta D (\lambda)$ takes values in $\Lambda \mathfrak{su} (2)_{\sigma}$ and thus 
 $\exp \left( i\theta D (\lambda) \right)$ takes values in $\Lambda \mathrm{SU} (2)_{\sigma}$.

  We now recall the definition of an equivariant surface, see for an example \cite{DW2024}: 
  Let $G$ be a real Lie group and $K$ a closed subgroup. Let $M$ denote a connected Riemann surface and $f : M \to G/K$ a differentiable map. Let $g_{t}$ be a one-parameter group of biholomorphic maps of $M$. Then $f$ is called \textit{$\mathbb R$-equivariant} with regard to $g_{t}$ if there exists a one-parameter group $R(t) : \mathbb{R} \to G$ such that 
    \begin{equation*}
        f(g_{t} \cdot p) = R(t) f(p)
    \end{equation*}
    for all $p \in M$ and $t \in \mathbb{R}$.
Then a straightforward computation shows that the minimal Lagrangian surface constructed 
by the equivariant potential $\xi$ in \eqref{eq: Delaunay potential} is an equivariant surface.
\begin{Proposition}
    Let $\xi$ be an $\mathbb R$-equivariant potential defined in \eqref{eq: Delaunay potential}
    and let $F_{\lambda} \in \Lambda \SU_{\sigma}$ for $\lambda \in \mathbb{S}^{1}$ be the corresponding
    extended frame.
    Then the minimal Lagrangian surface $f^{\lambda} : M \to Q_{2}$ constructed by $\left( F_{\lambda }, F_{-i\lambda} \right)$ is equivariant, that is,
    \[
     \hat {\mathfrak f}^{\lambda}\left( e^{i\theta} \cdot z , \, e^{-i\theta} \cdot \bar{z} , \, \lambda \right)=
     \psi\left((\exp(i\theta D(\lambda) ), \exp(-i\theta D(-i\lambda))\right)\hat{\mathfrak f}^{\lambda}(z, \bar{z}, \lambda) )
    \]
    holds, where $\hat {\mathfrak f}^{\lambda}$ is the horizontal lift of $f^{\lambda}$ in $\mathbb C^4$ and 
    $\psi : \Lambda \SU_\sigma \times \Lambda \SU_\sigma \to \Lambda \SO_\sigma$ is the loop group homomorphism.
\end{Proposition}

By Theorem 3 in \cite{SKKR}, we can obtain that 
\begin{equation*}
   F_{\lambda} = \begin{pmatrix}
 \sqrt{\frac{4ab\lambda^{2} +v^{2}}{2v(a\lambda^{2} + b)}}\cosh (tz - t\mathbf{f}) & \frac{\lambda v'\cosh(tz - t\mathbf{f})}{\sqrt{2v(a\lambda^{2} + b)(4ab\lambda^2+ v^{2})}}+\sqrt{\frac{2v(a + b\lambda^{2})}{4ab\lambda^2 + v^{2}}}\sinh (tz - t\mathbf{f}) \\
 \sqrt{\frac{4ab\lambda^2 + v^{2}}{2v(a + b\lambda^{2})}}\sinh(tz-t\mathbf{f}) & \frac{\lambda v'\sinh(tz - t\mathbf{f})}{\sqrt{2v(a + b\lambda^{2})(4ab\lambda^2 + v^{2})}}+\sqrt{\frac{2v(a\lambda^{2} + b)}{4ab\lambda^2 + v^{2}}}\cosh(tz-t\mathbf{f})
\end{pmatrix},
\end{equation*}
where $a , b \in \mathbb{R} \setminus \{0\}$ and $c =0$, and the non-constant Jacobian elliptic function $v = v(x)$ and $\mathbf{f} = \mathbf{f}(x)$ are given by 
\begin{equation*}
    v'^{2} = -(v^{2} - 4a^{2})(v^{2} - 4b^{2}), \quad v(0) = 2b, \quad \mathbf{f}(x) = \int^{x}_{0} \frac{2\text{d} s}{1 + (4ab\lambda^2)^{-1}v^{2}(s)},
\end{equation*}
and 
\begin{equation*}
    t = \sqrt{(a\lambda + b\lambda^{-1})(a\lambda^{-1} + b \lambda)}.
\end{equation*}
It is also easy to compute the map $F_{-i\lambda}$. Then by Proposition \ref{Prop}, we obtain the explicit form of this equivariant surface in $Q_{2}$.

 Note that the resulting $\mathbb R$-equivariant minimal Lagrangian surface $f^{\lambda}$ is 
 only defined on the universal cover of $\mathbb C \setminus \{0\}$ for a general choices of $a, b$ and $c$. 
 We now give the
 sufficient conditions such that $f^{\lambda}$ is well-defined on $\mathbb C \setminus \{0\}$, that is,
 $f^{\lambda}$ is a topologically cylinder:
 \begin{Proposition}\label{prp:closing}
 The resulting $\mathbb R$-equivariant minimal Lagrangian surface $f^{\lambda}|_{\lambda = \lambda_0}$ is well-defined on $\mathbb C \setminus \{0\}$ if the eigenvalues of $D(\lambda_0)$ and $D(- i \lambda_0)$ are integers or 
 half-integers.
  \end{Proposition}
\begin{Remark}
\mbox{}
\begin{enumerate}
    \item The  conditions in Proposition \ref{prp:closing}
    are explicitly given by 
    \[
    2 \sqrt{c^{2} + |\lambda_0 a + \lambda_0^{-1}b|^{2}}, \quad 2 \sqrt{ c^{2} + |\lambda_0 a - \lambda_0^{-1}b|^{2}} \in  \mathbb Z_{\geq 0}.           
    \] 
    \item In \cite[Section 8]{SKKR}, the conditions for the well-definedness on $\mathbb C \setminus \{0\}$
    of $\mathbb R$-equivariant constant mean curvature surfaces in 
    $\mathbb S^3$ have been given by that 
    the eigenvalues of $D(\lambda_0)$ and $D(- i \lambda_0)$ are half-integers.
    Thus even the corresponding $\mathbb R$-equivariant minimal surface is not well-defined,
    the minimal Lagrangian surface could be well-defined on $\mathbb C \setminus \{0\}$.
    Such cases exactly happen if the one of eigenvalues $D(\lambda_0)$ and $D(- i \lambda_0)$
    is an integer and the other is a half-integer, and thus
    \[
   \hat{\mathfrak{f}}^{\lambda}\left( e^{2 \pi i} \cdot z , \, e^{-2 \pi i} \cdot \bar{z} , \, \lambda_0 \right)
   =-\hat{\mathfrak{f}}^{\lambda}\left(z , \, \bar{z} , \, \lambda_0 \right)
    \]
    holds for the horizontal lift $\hat{\mathfrak{f}}^{\lambda}$, but it is the same point in $Q_2$.
\end{enumerate}

\end{Remark}
\subsubsection{Radially symmetric minimal Lagrangian surfaces}
\begin{Definition}[Radially symmetric potentials, \cite{FKR2009}]
Define 
\begin{equation}\label{eq: Smyth potential}
    \xi = \lambda^{-1} \begin{pmatrix}
 0 & 1 \\
 cz^{k} & 0
\end{pmatrix} \text{d} z, 
\end{equation}
for $z \in \mathbb{D} = \mathbb{C}$ and $ k \in \mathbb N$ and some $c \in \mathbb{C} \setminus (\mathbb{S}^{1} \cup \{ 0 \} ) $. 
Here we call such potentials the \textit{radially symmetric potentials}.
\end{Definition}
First we recall the result of \cite[Lemma 2.6.1]{FKR2009}:
The constant mean curvature surface in $\mathbb R^3$ (the so-called Smyth surfaces) constructed by $\xi$ 
has reflection symmetries with respect to $k+2$ geodesic planes that meet equiangularly along a geodesic line.

Thus it is natural to think the reflections 
$R_{\ell}(z) = e^{\frac{2\pi i \ell}{k+2}} \bar{z}$
of the domain $\mathbb{C}$, for $\ell \in \{ 0 , 1, \dots, k+1 \}$. Note that
\begin{equation}\label{eq:symxi}
    \xi (R_{\ell}(z), \lambda) = A_{\ell} \xi(\bar{z} , \lambda) A_{\ell}^{-1}
\end{equation}
holds, where
\begin{equation*}
    A_{\ell} = \begin{pmatrix}
 e^{\frac{\pi i \ell}{k+2}} & 0 \\
 0 & e^{\frac{-\pi i \ell}{k+2}}
\end{pmatrix} \in \SU.
\end{equation*}
Let $\Phi$ be the solution of $d \Phi = \Phi \xi$ with $\Phi(z_0) = \id$ and 
consider the Iwasawa decomposition $\Phi = F_{\lambda} B$.
Then by \eqref{eq:symxi}, 
\begin{equation*}
    F(R_{\ell}(z),\overline{R_{\ell}(z)}, \lambda) = A_{\ell} F(\bar{z}, z , \lambda) A_{\ell}^{-1}
\end{equation*}
holds. This leads to the following proposition.
\begin{Proposition}
    Let $\xi$ be the radially symmetric potentials defined in \eqref{eq: Smyth potential} and let $F_{\lambda} \in \Lambda\SU_{\sigma}$ for $\lambda \in \mathbb{S}^{1}$ be an extended frame obtained by $\xi$. Then the minimal Lagrangian surface $f^{\lambda}: M \to Q_{2}$ constructed by $( F_{\lambda}, F_{-i\lambda} )$ 
    has discrete rotational symmetries$:$
    \begin{equation*}
    \hat{\mathfrak f}^{\lambda}( R_{\ell}(z), \overline{R_{\ell}(z)}, \lambda )  = \mathcal{A}_{\ell} \hat{\mathfrak f}^{\lambda}(\bar{z} , z , \lambda), \quad
    \mbox{with} \quad \mathcal{A}_{\ell} :=\begin{pmatrix}
 1 & 0 & 0 & 0 \\
 0 & 1 & 0 & 0 \\
 0 & 0 & \cos \left( \frac{2\pi \ell}{k + 2} \right) & -\sin \left( \frac{2\pi \ell}{k + 2} \right) \\
 0 & 0 & \sin \left( \frac{2\pi \ell}{k + 2} \right) & \cos \left( \frac{2\pi \ell}{k + 2} \right)
\end{pmatrix},
\end{equation*}
 where $\hat{\mathfrak f}^{\lambda}$ is the horizontal lift of $f^{\lambda}$.
 Moreover, the metric of the minimal Lagrangian surface $f^{\lambda}$ depends only on the radial direction, that is $|z|$. 
Such a minimal Lagrangian surface $f^{\lambda}$ is said to be {\rm radially symmetric}.
\end{Proposition}
\begin{proof}
    Let $\psi$ be a loop group homomorphism 
    $\psi : \Lambda \SU_{\sigma} \times \Lambda \SU_{\sigma} \to \Lambda \SO_{\sigma}$ as before, and 
    set $\hat{\mathcal F}(z, \bar z, \lambda) = \psi (F(z, \bar z, \lambda), F(z, \bar z, - i \lambda))$.
    By direct computation, we have
\begin{equation*}
  \begin{aligned}
      \hat{\mathcal F}( R_{\ell}(z), \overline{R_{\ell}(z)}, \lambda ) &= \psi  \left( F(R_{\ell}(z) , \overline{R_{\ell}(z)} , \lambda ) , F( R_{\ell}(z) , \overline{R_{\ell}(z)} , -i\lambda ) \right)  \\
      &= \psi \left( \left( A_{\ell} , A_{\ell} \right) \cdot \left( F(\bar{z}, z , \lambda), F(\bar{z}, z, -i\lambda) \right) \cdot \left( A_{\ell} , A_{\ell} \right)^{-1}   \right) \\
      &= \psi \left(  A_{\ell} , A_{\ell}  \right) \cdot \hat{\mathcal F}(\bar{z} , z , \lambda) \cdot  \left( \psi  \left( A_{\ell} , A_{\ell} \right) \right)^{-1},
  \end{aligned}
\end{equation*}
where 
\begin{equation*}
    \psi\left( A_{\ell}, A_{\ell} \right) = \begin{pmatrix}
 1 & 0 & 0 & 0 \\
 0 & 1 & 0 & 0 \\
 0 & 0 & \cos \left( \frac{2\pi \ell}{k + 2} \right) & -\sin \left( \frac{2\pi \ell}{k + 2} \right) \\
 0 & 0 & \sin \left( \frac{2\pi \ell}{k + 2} \right) & \cos \left( \frac{2\pi \ell}{k + 2} \right)
\end{pmatrix}.
\end{equation*}
 Then a straightforward computation shows the first claim. 
 
 It is shown in \cite{BI1995} that the real function $\hat{u}(z,\bar{z})$ in the metric $2e^{\hat{u}}\mathrm{d} z \mathrm{d} \bar{z}$ of the surface from the frame $F_{\lambda}$ is independent of $\arg (z)$.  By \eqref{eq: metric}, the real function $u(z, \bar{z})$ in the metric $2e^{u}\mathrm{d} z \mathrm{d} \bar{z}$ of the minimal Lagrangian surface $f^{\lambda}$ only depends on $|z|$.
This completes the proof.
\end{proof}
 Note that when $c \in \mathbb{C} \setminus (\mathbb{S}^{1} \cup \{ 0 \} )$, there is no simple way to perform 
the Iwasawa decomposition of this solution $\phi$, so we cannot write the explicit form of all radially symmetric surfaces. 

\subsection{Minimal Lagrangian trinoids}
Finally, we construct minimal Lagrangian surfaces from thrice-punctured sphere into $Q_2$
according to \cite{SKKR}. We first discuss the \textit{monodromy problem} of the DPW method:
For a given holomorphic potential $\xi$, there exists a unique solution $\Phi$  in 
\textbf{Step 1} in Section \ref{sbsc:DPW}. Note that even $\xi$ is well-defined on a 
Riemann surface $M$, the solution $\Phi$ may not be well-defined on $M$, that is, $\Phi$
is well-defined on the universal cover of $M$ (see Section \ref{subsub:equiv} for an example). To simplify the exposition, let us take 
$M = \mathbb {C}P^1 \setminus \{0, 1, \infty\}$, and set simple loops starting from the base point 
$z_0 \neq 0, 1, \infty$ going around $0, 1$ and $\infty$ ending to $z_0$
by $\gamma_j (j=0, 1, \infty)$.  Note that it has a simple relation 
 $\gamma_{\infty} \gamma_1 \gamma_0 = \id$,
and $\gamma_i$ and $\gamma_j (i \neq j)$  are the generators of $\pi_1 (M)$. The fundamental 
relations are given as follows:
\[
 \gamma_j^*  \Phi = \mathcal H (\gamma_j) \Phi, \quad \mathcal H (\gamma_j) \in \Lambda \SL_{\sigma},
\]
and from the relation of $\gamma_{\infty} \gamma_1 \gamma_0 = \id$,  $\mathcal H (\gamma_{\infty})\mathcal H (\gamma_1)\mathcal H (\gamma_0) = \id$ follows. The matrices 
$\mathcal H (\gamma_j)$ are called the \textit{monodromy matrices}. In \textbf{Step 2}, we 
have considered the 
Iwasawa decomposition to obtain the extended frame $F_{\lambda}$. In general,
the behavior of $F_{\lambda}$ under $\gamma_j$ is not obvious, since the monodromy matrix
$\mathcal H (\gamma_j)$ does not take values in $\Lambda \SU_{\sigma}$. Moreover, even though 
$\mathcal H (\gamma_j)$ takes values
$\Lambda \SU_{\sigma}$, it is not clear that the resulting minimal Lagrangian immersion $f^{\lambda}$
in \textbf{Step 4} is well-defined on $M$. We now state the following sufficient condition.
\begin{Proposition}\label{prp:closing2}
 The resulting minimal Lagrangian immersion $f^{\lambda}|_{\lambda = \lambda_0}$ is 
 well-defined on $M = \mathbb {C}P^1 \setminus \{0, 1, \infty\}$ if the monodromy matrices $\mathcal H (\gamma_j) (j=0, 1, \infty)$  satisfy the following two conditions:
 \begin{enumerate}
     \item\label{item1} $\mathcal H (\gamma_j)$ takes values in $\Lambda \SU_{\sigma}$  for $j=0, 1, \infty$.
     \item\label{item2} $\mathcal H (\gamma_j)|_{\lambda = \lambda_0} = \mathcal H (\gamma_j)|_{\lambda = -i \lambda_0}= \pm \id$.
 \end{enumerate}
\end{Proposition}
\begin{proof}
    From the assumption \eqref{item1}, it is easy to see 
    \[
    \gamma_j^* F_{\lambda} =  \mathcal H (\gamma_j) F_{\lambda}
    \]
    holds.
   Then the resulting minimal Lagrangian immersion $f^{\lambda}|_{\lambda = \lambda_0}$ 
   is given by \eqref{eq:Xlambda} and \eqref{eq:formula}, and the assumption \eqref{item2} 
   clearly shows that 
  \[
  \gamma_j^* f^{\lambda}|_{\lambda= \lambda_0} = f^{\lambda}|_{\lambda= \lambda_0}
  \]
  holds. This completes the proof.
\end{proof}
\begin{Remark}
 From the relation $\mathcal H (\gamma_{\infty})\mathcal H (\gamma_1)\mathcal H (\gamma_0) =\id$, the assumptions of Proposition \ref{prp:closing2}
 only need to satisfy two matrices of $\{\mathcal H (\gamma_0), \mathcal H (\gamma_1), 
 \mathcal H (\gamma_{\infty})\}$.
\end{Remark}

Now we give the definition of trinoid potential, which can be used to construct constant mean curvature 
trinoid in $\mathbb{S}^{3}$.  In the following, we use the \textit{untwisted} loop group 
according to \cite{SKKR}, however, it makes no essential difference.
Indeed, there is an isomorphism between them:
For an untwisted element $g (\lambda) \in \Lambda \SU$, define a map
\[
 g(\lambda)  \mapsto \operatorname{Ad}\operatorname{diag} (\sqrt{\lambda}, 
 \sqrt{\lambda}^{-1}) g(\lambda^2),
\]
 which gives a twisted element $\Lambda \SU_{\sigma}$
 and it is an isomorphism between the loop groups.
\begin{Definition}[Minimal trinoid potentials, \cite{SKKR}]
    Let $M = \mathbb{C}P^{1}\setminus \left\{ 0 , 1, \infty \right\}$ be the thrice-punctured Riemann sphere. Let $\lambda_{0} \in \{ \pm \sqrt{-1}\}$, and let 
    \begin{equation*}
        h(\lambda) = \lambda^{-1}(\lambda - \lambda_{0})(\lambda - \lambda_{0}^{-1}).
    \end{equation*}
    Let $v_{0}, v_{1}, v_{\infty} \in \mathbb{R}\setminus \{ 0 \}$. The family of trinoid potentials $\xi$, parametrized by $\lambda_{0}$ and $v_{0}, v_{1}, v_{\infty}$, is defined by 
    \begin{equation*}
        \xi = \begin{pmatrix}
 0 & \lambda^{-1}\mathrm{d}z \\
 \lambda h(\lambda) Q/ \mathrm{d}z & 0
\end{pmatrix}, \quad Q = \frac{v_{\infty}z^{2} + (v_{1} - v_{0} - v_{\infty})z + v_{0}}{16z^{2}(z - 1)^{2}} \mathrm{d}z^{2}.
    \end{equation*}
    Here $Q$ will be the Hopf differential of the resulting minimal trinoid, up to a $z$-independent multiplicative constant. Its three poles are the ends of the minimal trinoid, and its two zeros the umbilic points.
\end{Definition} 
\begin{Remark}
 The choice of $\lambda_0$ determines the value of constant mean curvature of the resulting 
 trinoids in $\mathbb S^3$, and we are only interested in minimal trinoids; thus, we choose 
 $\lambda_0 \in \{ \pm \sqrt{-1}\}$.
\end{Remark}

\begin{Theorem}[Theorem 5 in \cite{SKKR}]\label{Thm trinoid S3}
    Let $\mathcal H (\gamma_j) ( j=0,1, \infty)$ be monodromy matrices of a trinoid potential, and assume that $\mathcal H (\gamma_j)$ is holomorphic on $\mathbb{C}^{\ast}$. Let the parameters in the potential be $\lambda_{0}, v_{0}, v_{1}, v_{\infty}$, and for $k \in \{ 0 , 1, \infty \}$ suppose that $1 + v_{k}h(-1)/4 \ge 0 $ and $1 + v_{k}h(1)/4 \ge 0$. Define
    \begin{equation*}
        n_{k} = \frac{1}{2}- \frac{1}{2} \sqrt{1 + v_{k}h(-1)/4}, \quad m_{k} = \frac{1}{2} - \frac{1}{2} \sqrt{1 + v_{k}h(1) /4}.
    \end{equation*}
    Suppose the following inequalities hold for every permutation $(i , j , k)$ of $(0 , 1 , \infty)$: 
    \begin{equation*}
        \begin{aligned}
            &|n_{0}| + |n_{1}| + |n_{\infty}| \le 1 \quad \text{and} \quad |n_{i}| \le |n_{j}| + |n_{k}|, \\
            &|m_{0}| + |m_{1}| + |m_{\infty}| \le 1 \quad \text{and} \quad |m_{i}| \le |m_{j}| + |m_{k}|.
        \end{aligned}
    \end{equation*}
    Then such a trinoid potential satisfying the above conditions generates via the DPW method
    a conformal minimal immersion of the thrice-punctured Riemann sphere into $\mathbb{S}^{3}$.
\end{Theorem}
 Applying Proposition \ref{prp:closing2}, we obtain the following corollary.
\begin{Corollary}
 The trinoid potentials $\xi$ with the assumption in {\rm Theorem \ref{Thm trinoid S3}}
 yield minimal Lagrangian trinoids in $Q_2$.
\end{Corollary}

\appendix

\section{Comparison to the result of Castro-Urbano}\label{app: CU}

In this section, we shall discuss a relation to the result of Castro-Urbano \cite{CU} and establish a correspondence with their quantities. The following notations and definitions are all used in 
\cite{CU}.

Let $\mathbb{S}^{2}$ be the unit sphere in the Euclidean space $\mathbb{R}^{3}$ endowed with its standard Euclidean metric $\left\langle, \right\rangle$ and its structure of Riemann surface given by $J_{x}v = x \times v$. for any $v \in T_{x}\mathbb{S}^{2}$, $x \in \mathbb{S}^{2}$, where $\times$ stands for the vectorial product in $\mathbb{R}^{3}$.  Its K\"ahler 2-form is the area 2-form $\omega_{0}$ defined by 
\begin{equation*}
    \omega_{0} \left( v , w \right) = \langle J_{x}v, w \rangle = \det \{ x, v , w  \}
\end{equation*}
for any $v, w \in T_{x}\mathbb{S}^{2}$. We endow $\mathbb{S}^{2} \times \mathbb{S}^{2}$ with the product metric (also denoted by $\left\langle, \right\rangle$) and the product complex structure $J$ given by
\begin{equation*}
    J_{(x, y)}(v) = \left( J_{x}v_{1} , J_{y}v_{2} \right) = \left( x \times v_{1}, y \times v_{2} \right), 
\end{equation*}
for any $v = \left( v_{1} , v_{2} \right) \in T_{(x , y)}(\mathbb{S}^{2} \times \mathbb{S}^{2})$, $(x , y) \in \mathbb{S}^{2} \times \mathbb{S}^{2}$, which becomes $\mathbb{S}^{2} \times \mathbb{S}^{2}$ in a K\"ahler surface. Its K\"ahler 2-form is $\omega = \pi^{\ast}_{1}\omega_{0} + \pi^{\ast}_{2}\omega_{0}$, where $\pi_{i}$, $i = 1,2$, are the projections of $\mathbb{S}^{2} \times \mathbb{S}^{2}$ onto $\mathbb{S}^{2}$.

Let $\Phi = \left( \phi, \psi \right) : \Sigma \to \mathbb{S}^{2} \times \mathbb{S}^{2}$ be a Lagrangian immersion of an oriented surface with
area 2-form $\omega_{\Sigma}$. Define the Jacobians of $\phi$ and $\psi$ by
\begin{equation*}
    \phi^{\ast}\omega_{0} = \mathrm{Jac} (\phi) \omega_{\Sigma}, \quad \psi^{\ast}\omega_{0} = \mathrm{Jac} (\psi) \omega_{\Sigma}, \quad
    \mathrm{Jac} (\phi) = -\mathrm{Jac} (\psi).
\end{equation*}
We will call the function 
\begin{equation*}
    C : =  \mathrm{Jac} (\phi) = -\mathrm{Jac} (\psi)
\end{equation*}
the \textit{associated Jacobian} of the oriented Lagrangian surface $\Sigma$.

We consider a local isothermal parameter $z = x + iy$ on $\Sigma$. 
To begin with, we assume the metric 
\begin{equation}\label{eq:metricofS2}
ds^{2}_{\Sigma} = 8e^{u}\text{d}z\text{d}\bar{z},
\end{equation} 
and the formulas for Lagrangian immersion $\Phi$ in $\mathbb{S}^{2} \times \mathbb{S}^{2}$ transform into 
\begin{equation*}
\begin{aligned}
    &|\Phi_{z}|^{2} = 2|\phi_{z}|^{2} = 2|\psi_{z}|^{2} = 4e^{u}, \quad \Phi_{z\bar{z}} = 4e^{u}\left( H - \frac{1}{2}\Phi \right), \\
    &\Phi_{zz} = u_{z}\Phi_{z} + \langle H , J\Phi_{z} \rangle J \Phi_{z} + \frac{1}{4}e^{-u} \langle \Phi_{zz}, J \Phi_{z} \rangle J\Phi_{\bar{z}} - \frac{1}{2}\langle \Phi_{z} , \hat{\Phi}_{z}\rangle \hat{\Phi}, \\
    &\hat{\Phi}_{z} = \frac{1}{4}e^{-u}\langle \Phi_{z}, \hat{\Phi}_{z}  \rangle \Phi_{\bar{z}} - 2iCJ\Phi_{z},
\end{aligned}
\end{equation*}
where the derivatives respect to $z$ and $\bar{z}$ are given by $\partial_z = \tfrac12 (\partial_x - i \partial_y)$, $\partial_{\bar z} = \tfrac12 (\partial_x + i \partial_y)$, and $H$ is the mean curvature of $\Phi$, and $\hat{\Phi} = (\phi , -\psi)$. 
These quantities correspond to (4.3), (4.4) and (4.6) in \cite{CU}. When $\Phi$ is minimal, it follows that
\begin{align}
    &\Phi_{z\bar{z}} = -2e^{u}\Phi, \label{eq: CU 1} \\
    &\Phi_{zz} = u_{z}\Phi_{z} + \frac{1}{4}e^{-u} \langle \Phi_{zz}, J \Phi_{z} \rangle J\Phi_{\bar{z}} -\frac{1}{2}\langle \Phi_{z} , \hat{\Phi}_{z} \rangle\hat{\Phi}, \\
    &\hat{\Phi}_{z} = \frac{1}{4}e^{-u}\langle \Phi_{z} , \hat{\Phi}_{z} \rangle \Phi_{\bar{z}} - 2iCJ\Phi_{z}, \\
    &16e^{2u}(1 - 4C^{2}) = \left|\langle \Phi_{z} , \hat{\Phi}_{z} \rangle\right|^{2} = 4\left| \left\langle \phi_{z} , \phi_{z} \right\rangle \right|^{2} = 4\left| \left\langle \psi_{z}, \psi_{z}  \right\rangle \right|^{2}. \label{eq: CU 4}
\end{align}
The equation \eqref{eq: CU 1} implies that $\phi_{z\bar{z}} = -2e^{u}\phi$ and $\psi_{z\bar{z}} = -2e^{u}\psi$. This means that $\phi, \psi : \Sigma \to \mathbb{S}^{2}$ are harmonic maps. Thus, the associated Hopf differential
\begin{equation*}
    \Theta(z) = \left\langle \phi_{z}, \phi_{z} \right\rangle \otimes (\mathrm{d} z)^{2} = -\left\langle \psi_{z}, \psi_{z} \right\rangle \otimes (\mathrm{d} z)^{2} = \frac{1}{2} \langle \Phi_{z}, \hat{\Phi}_{z} \rangle \otimes (\mathrm{d} z)^{2}
\end{equation*}
is holomorphic. From \eqref{eq: CU 4}, we have 
\begin{equation}\label{eq: CU Godazzi}
    |\Theta|^{2} = 4e^{2u} \left(1 - 4 C^{2}\right).
\end{equation}
 The squared norm of the gradient of $C$ is given in (4.12) in \cite{CU} by
 \[
 \left|\nabla C \right|^{2} = \frac{\left(1- 4 C^2\right) \left( 2C^2 -K\right)}{2}. 
 \]
 Since the Gauss curvature $K$ can be given by $K = -\frac14 e^{-u}u_{z\bar{z}}$, it can be rephrased as 
 \begin{equation}\label{eq:generalGC1}
4|C_{z}|^{2}= \left(1- 4 C^2\right) \left(u_{z\bar{z}}  + 8e^{u} C^{2}\right). 
\end{equation}        

 To establish the correspondence between minimal Lagrangian surfaces in $Q_{2}$ and $\mathbb{S}^{2} \times
 \mathbb{S}^{2}$, 
 we first note that $Q_2$ is isometric to $\mathbb S^2 \times \mathbb S^2$, where 
 two $2$-spheres $\mathbb S^2$ have the constant curvature $4$ \cite{WV2021}. Thus to compare our surface $f$ and 
 \cite{CU}, we need to use the metric
  \[
 4 ds^{2}_{M} =  ds_{\Sigma}^2 = 8e^{u}\mathrm{d}z\mathrm{d}\bar{z}.
 \]
 
 We now show \eqref{eq:generalGC1} is the Gauss equation for a minimal Lagrangian 
 surface $f$ in $Q_2$ under suitable identification.
\begin{Proposition}\label{Cor1}
    Let $f : M \to Q_{2}$ be a minimal Lagrangian immersion, and define 
    $\hat C = \frac{1}{2}e^{-u} |\beta|$.  Then the last PDE of \eqref{eq: M-C} is equivalent to 
\begin{equation}\label{eq:generalGC2}
    u_{z\bar{z}} + 8e^{u}\hat C^{2} - \frac{4|\hat{C}_{z}|^{2}}{1-4\hat{C}^{2}}= 0.
\end{equation}        
    \end{Proposition}
\begin{proof}
 Note $|\alpha|^{2} = e^{2u}(1 - 4\hat C^{2})$ holds.  
 When $\beta \equiv 0$, that is, $\hat C=0$, by Lemma \ref{lem:totally}, $f$ becomes the totally geodesic torus and 
 \eqref{eq:generalGC2} clearly holds. From now on we assume that $\beta \not\equiv 0$.
By $\hat C = \frac{1}{2}e^{-u} |\beta|$, it follows that
\begin{equation*}
    \frac{1}{4}|\alpha_{z}|^{2}e^{u} + |u_{z}|^{2} e^{u} |\alpha|^{2} - \frac{1}{2} \alpha_{z} u_{\bar{z}} e^{u} \bar{\alpha} - \frac{1}{2}\bar{\alpha}_{\bar{z}} u_{z} e^{u} \alpha = e^{u}|\beta|^{2} \left( \frac{4|\hat{C}_{z}|^{2}}{1-4\hat{C}^{2}} \right),
\end{equation*}
and 
\begin{equation*}
2|\beta|^{4} = e^{u} |\beta|^{2} \left( 8e^{u}\hat C^{2} \right).
\end{equation*}
Note that 
$e^{u}|\beta|^{2} \left( \frac{4|\hat{C}_{z}|^{2}}{1-4\hat{C}^{2}} \right)$ is well-defined even if a point $p$ where $\hat C(p) = \frac{1}{2}$, since 
$\hat C_z(p)=0$ there.
Thus the last PDE of \eqref{eq: M-C} can be written as follows:
\begin{equation}\label{eq: second order PDE1}
    e^{u}|\beta|^{2}\left( u_{z\bar{z}} + 8e^{u}\hat C^{2} - \frac{4|\hat{C}_{z}|^{2}}{1-4\hat{C}^{2}} \right) = 0.
\end{equation}
Since $\beta \not\equiv 0$, we have \eqref{eq:generalGC2}. This completes the proof.
\end{proof}
 By Proposition \ref{Cor1}, it is natural to have the 
 correspondence of the associated Hopf differential $\Theta$ and the associated Jacobian $C$ 
with the quantities $\alpha, \beta$ and $u$ by
\begin{equation}\label{eq: CU alpha}
    \Theta = 2\alpha, \quad     C  = \hat C = \frac{1}{2}e^{-u} |\beta|.
\end{equation}
It is easy to check the condition $C^{2} < 1/4$ in \cite[Theorem 4.4]{CU} is actually the condition $|\alpha| > 0$. 
 Indeed, a straightforward computation shows that \eqref{eq: CU alpha} holds for surfaces 
 $\Phi$ in $\mathbb S^2  \times \mathbb S^2$ and  $f$ in $Q_2$.
 \begin{Remark}
  The normalization of $\Theta$ given in \cite{CU} corresponds to the normalization of $\Theta$ here, i.e. $|\Theta| = 1$. We then obtain the Gauss and Codazzi equations (4.19) in \cite{CU} of the minimal Lagrangian immersion $\Phi$ under the new metric $ds^{2}_{\Sigma} = 8e^{u}\text{d}z\text{d}\bar{z}$ as follows:
\begin{equation}\label{eq: CU Gauss-Codazzi}
    u_{z\bar{z}} + 8e^{u}C^{2} - 16e^{2u}|C_{z}|^{2} = 0, \quad 1 - 4C^{2} = \frac{1}{4}e^{-2u}.
\end{equation}
 \end{Remark}

\section{The Lie group and the Lie algebra isomorphism for $\SO$ }\label{app: lie group lie algebra}
 We first check that the Lie group isomorphism
\begin{equation}\label{Lie group isomorphism1}
 \SO \cong  (\SU \times \SU)/\boldsymbol{Z}_2,    
\end{equation}
 which is explicitly given in \cite{Yokota}.

The set of all quaternions can be defined by 
\begin{equation*}
\begin{aligned}
    \bsH = \bsR \oplus \bsR \bsi \oplus \bsR \bsj \oplus \bsR \bsk= \{ x + y \bsi + s\bsj + t\bsk ~|~ x, y, s, t \in \bsR \},
\end{aligned}
\end{equation*}
with $\bsi^{2} = \bsj^{2} = \bsk^{2} = -1, 
    \bsi \bsj = -\bsj \bsi = \bsk, \bsj \bsk = -\bsk \bsj = \bsi,  \bsk \bsi = -\bsi \bsk = \bsj
    $, where $\bsR$ is the set of all real numbers.
For a quaternion $\alpha$, it can be expressed as follows:
\begin{equation*}
    \alpha = x + y\bsi +  s\bsj + t\bsk = (x + y\bsi) + (s + t\bsi)\bsj \quad \quad x, y, s, t \in \bsR.
\end{equation*}
Thus the set $\bsH$ can also be written as
$\bsH = \bsC \oplus \bsC \bsj = \{ a + b\bsj ~|~ a, b \in \bsC \}$ with $\bsj^{2} = -1$,
where $\bsC$ is the set of all complex numbers. Now we define the symplectic group by 
\begin{equation*}
    {\rm{Sp}} (n) := \{ A \in M(n , \bsH) ~|~ A^{\ast}A = E\}. 
\end{equation*}
In particular, if $\alpha = a + b\bsj \in \Sp$, it satisfies $|a|^{2} + |b|^{2} = 1$.

The Lie group isomorphism \eqref{Lie group isomorphism1} can be established in the following two steps.
\begin{enumerate}
    \item $\SU \cong \Sp$. When $p = p_{0} + p_{1} \bsi + p_{2} \bsj + p_{3}\bsk \in \Sp$, 
    we can establish the map $k : \Sp \to \SU$ by
   \begin{equation*}
    k(p) = \begin{pmatrix}
 p_{0} + p_{1} \bsi & p_{2} + p_{3}\bsi \\
 -p_{2} + p_{3}\bsi & p_{0} - p_{1} \bsi
\end{pmatrix}. 
\end{equation*}
See \cite[p. 93, Theorem 10.4]{Yokota} for details. 

    \item $(\Sp \times \Sp) / \boldsymbol{Z}_{2} \cong \SO$. When $p = p_{0} + p_{1} \bsi + p_{2} \bsj + p_{3}\bsk,~ q = q_{0} + q_{1} \bsi + q_{2} \bsj + q_{3}  \bsk \in \Sp$, we can establish the map $\psi : \Sp \times \Sp \to \SO$ by
    \begin{equation}\label{eq:psi}
       \psi (p , q) = \begin{pmatrix}
 p_{0} & -p_{1}  & -p_{2} & -p_{3} \\
 p_{1} & p_{0} & -p_{3} & p_{2} \\
 p_{2} & p_{3} & p_{0} & -p_{1} \\
 p_{3} & -p_{2} & p_{1} & p_{0}
\end{pmatrix} \begin{pmatrix}
 q_{0} & q_{1} & q_{2} & q_{3} \\
 -q_{1} & q_{0} & -q_{3} & q_{2} \\
 -q_{2} & q_{3} & q_{0} & -q_{1} \\
 -q_{3} & -q_{2} & q_{1} & q_{0}
\end{pmatrix}.
    \end{equation}
Clearly, the kernel of $\psi$ is $\pm \id$, and the isomorphism can be established.
See \cite[p. 99, Theorem 10.7]{Yokota} for details.  
\end{enumerate}

Accordingly, we can introduce the Lie algebra isomorphism  
\begin{equation}\label{Lie algebra isomorphism1}
\mathfrak{so}(4) \cong \mathfrak{su}(2) \oplus \mathfrak{su}(2).
\end{equation}
Define the Lie algebra of the symplectic group by
\begin{equation*}
    \mathfrak{sp}(n) := \{ X \in M(n , \bsH) ~|~ X^{\ast} + X = 0 \}.
\end{equation*}
In particular, if $\alpha = a + b\bsj \in \mathfrak{sp}(1)$, it is easy to check $\bar{\alpha} = \bar{a} - b \bsj$, and it follows that $\bar{a} = -a$.

The Lie algebra isomorphism \eqref{Lie algebra isomorphism1} can be established in the following two steps.
\begin{enumerate}
    \item $\mathfrak{sp}(1) \cong \mathfrak{su}(2)$. 
    When $p = p_{1} \bsi + p_{2} \bsj + p_{3}\bsk \in \mathfrak{sp}(1)$, we can establish the map $k : \mathfrak{sp}(1) \to \mathfrak{su}(2)$ by
    \begin{equation*}
    k(p) = \begin{pmatrix}
 p_{1}\bsi & p_{2} + p_{3}\bsi \\
 -p_{2} + p_{3}\bsi & -p_{1}\bsi
\end{pmatrix}.
\end{equation*}
See \cite[p. 119, Theorem 11.9]{Yokota} for details.

\item $\mathfrak{o}(4) \cong \mathfrak{sp}(1) \oplus \mathfrak{sp}(1)$. When $p = p_{1} \bsi + p_{2} \bsj + p_{3}\bsk,~ q = q_{1} \bsi + q_{2} \bsj + q_{3}  \bsk \in \mathfrak{sp}(1)$, we can establish the map $\psi : \mathfrak{sp}(1) \oplus \mathfrak{sp}(1) \to \mathfrak{o}(4)$ by
\begin{equation*}
    \psi(p_{1} \bsi + p_{2} \bsj + p_{3}\bsk, q_{1} \bsi + q_{2} \bsj + q_{3}  \bsk)
= \begin{pmatrix}
 0 & -p_{1} + q_{1} & -p_{2} + q_{2} & -p_{3} + q_{3} \\
 p_{1} - q_{1} & 0 & -p_{3} - q_{3} & p_{2} + q_{2} \\
 p_{2} - q_{2} & p_{3} + q_{3} & 0 & -p_{1} - q_{1} \\
 p_{3} - q_{3} & -p_{2} - q_{2} & p_{1} + q_{1}  & 0
\end{pmatrix}.
\end{equation*}
See \cite[p. 120, Theorem 11.10]{Yokota} for details.
\end{enumerate}

On behalf of all authors, the corresponding author states that there is no conflict of
interest. No new data were created or analyzed in this study.

\bibliographystyle{plain}
\bibliography{mybib}

\end{document}